\documentclass[12pt, a4paper]{amsart}
\usepackage{amscd,amsmath,amssymb,amsthm,amsfonts}
\usepackage[alphabetic]{amsrefs}
\usepackage{mathrsfs}
\usepackage[shortlabels]{enumitem}
\usepackage[all]{xy}

\usepackage{xcolor}
\usepackage[hypertexnames=true,colorlinks = true, allcolors=blue,linktoc = all, pdffitwindow = false, urlbordercolor = white]{hyperref}%

\def\Aut{\operatorname{Aut}}

\def\id{\operatorname{id}}

\def\Ad{\operatorname{Ad}}

\def\max{\operatorname{max}}

\def\id{\operatorname{id}}

\def\C{\mathbb{C}}

\def\R{\mathbb{R}}
\def\N{\mathbb{N}}
\def\Z{\mathbb{Z}}

\newcommand{\IC}[0]{\mathbb{C}}

 \newcommand{\IN}[0]{\mathbb{N}}
 \newcommand{\IP}[0]{\mathbb{P}}
 \newcommand{\IR}[0]{\mathbb{R}}
 \newcommand{\IT}[0]{\mathbb{T}}

 \newcommand{\IZ}[0]{\mathbb{Z}}

\newcommand{\CA}[0]{\mathcal{A}} 
 \renewcommand{\CD}[0]{\mathcal{D}}
\newcommand{\CE}[0]{\mathcal{E}} \newcommand{\CF}[0]{\mathcal{F}}
 \newcommand{\CH}[0]{\mathcal{H}}
 
 \newcommand{\CL}[0]{\mathcal{L}}
 \newcommand{\CN}[0]{\mathcal{N}}
\newcommand{\CO}[0]{\mathcal{O}} 
\newcommand{\CQ}[0]{\mathcal{Q}} 
 \newcommand{\CT}[0]{\mathcal{T}}
\newcommand{\CU}[0]{\mathcal{U}} 
 
 \newcommand{\CZ}[0]{\mathcal{Z}}

\newcommand{\FA}[0]{\mathfrak{A}} \newcommand{\FB}[0]{\mathfrak{B}}

\def\gxp{G \rtimes_\theta P}
\def\gpt{G,P,\theta}

\newtheorem{thm}{Theorem}[section]
\newtheorem{corollary}[thm]{Corollary}
\newtheorem{lemma}[thm]{Lemma}

\newtheorem{proposition}[thm]{Proposition}

\theoremstyle{definition}
\newtheorem{definition}[thm]{Definition}

\theoremstyle{remark}
\newtheorem{remark}[thm]{Remark}

\newtheorem{example}[thm]{Example}

\numberwithin{equation}{section}

\begin{document}

\title[The inner structure of boundary quotients of right LCM semigroups]{The inner structure of boundary quotients of right LCM semigroups}

\author[V.~Aiello]{Valeriano Aiello}
\address{
Section de Math\'  ematiques, Universit\' e de Gen\` eve, 2-4 rue du Li\` evre, Case Postale 64,
1211 Gen\` eve 4, Switzerland}
\email{valerianoaiello@gmail.com}
%

\author[R.~Conti]{Roberto Conti}
\address{Dipartimento di Scienze di Base e Applicate per l'Ingegneria \\ Sapienza Universit\`{a} di Roma \\ Italy}
\email{roberto.conti@sbai.uniroma1.it}

\author[S.~Rossi]{Stefano Rossi}
\address{Dipartimento di Matematica \\ Universit\`{a} di Roma Tor Vergata \\ Italy}
\email{rossis@mat.uniroma2.it}

\author[N.~Stammeier]{Nicolai Stammeier}
\address{Department of Mathematics \\ University of Oslo \\ P.O. Box 1053 \\ Blindern \\ NO-0316 Oslo \\ Norway}
\email{nicolsta@math.uio.no}

\begin{abstract}
We study distinguished subalgebras and automorphisms of boundary quotients arising from algebraic dynamical systems $(\gpt)$. Our work includes a complete solution to the problem of extending Bogolubov automorphisms from the Cuntz algebra in $2 \leq p<\infty$ generators to the $p$-adic ring $C^*$-algebra. For the case where $P$ is abelian and $C^*(G)$ is a maximal abelian subalgebra, we establish a picture for the automorphisms of the boundary quotient that fix $C^*(G)$ pointwise. This allows us to show that they form a maximal abelian subgroup of the entire automorphism group. The picture also leads to the surprising outcome that, for integral dynamics, every automorphism that fixes one of the natural Cuntz subalgebras pointwise is necessarily a gauge automorphism. Many of the automorphisms we consider are shown to be outer.


\end{abstract}

\date{\today}
\maketitle
{\hypersetup{linkcolor=black}\tableofcontents}

\section{Introduction}
Introduced as long ago as the late 1970s, \cite{Cun}, the Cuntz algebras have since been an undeniably interesting field of research. A particularly fascinating area within this field is the structure of the endomorphisms and the automorphisms of Cuntz algebras, which provide fertile grounds for a deep interplay between $C^*$-algebra theory, ergodic theory, and combinatorics. Over the last decade, remarkable progress has been made towards a better understanding of certain key features and challenging problems, see \cites{Conti, CRS, ConSz, CHS1, CHS2, CHS3, CHS4, CHS5} for a brief selection.

Inspired by \cite{Cun}, a wealth of constructions of $C^*$-algebras associated with various sorts of input data has been constructed. It is thus quite natural to ask to what extent the endomorphism structure of these $C^*$-algebras resembles the case of the original Cuntz algebras. For instance, through the works \cites{CuntzQ,CL,CLintegral2,Li0,Li1,CDL13,BRRW} there are classes of $C^*$-algebras that are associated with purely algebraic objects, such as rings, integral domains, fields, and arbitrary left cancellative semigroups. Especially the study of the $C^*$-algebras constructed in \cite{Li1} has led to an extensive list of impressive results, from which we would like at least to mention \cites{Li2,CEL,CEL2,ELR,ABLS}, and refer the reader to \cite{Cuntz-review} for a recent survey as well as to \cite{CELY} for a more detailed exposition with an emphasis on K-theory.

In \cite{LarsenLi}, Larsen and Li performed a detailed case study on the representation theory of one particular $C^*$-algebra of the former type: the $2$-adic ring $C^*$-algebra $\CQ_2$. This algebra is the universal $C^*$-algebra generated by a unitary $u$ and an isometry $s$ subject to $su=u^2s$ and $ss^* +uss^*u^* = 1$, which can also be described as the ring $C^*$-algebra associated to $\Z \rtimes \langle 2 \rangle \subset \Z \rtimes \Z^\times$. One reason for choosing $\CQ_2$ is that the algebra is a balanced version of $\CO_2$ in the sense that it is a unital UCT Kirchberg algebra whose K-groups are both equal to $\Z$. In addition, $\CO_2$ appears quite naturally as the subalgebra $C^*(\{s,us\})$ of $\CQ_2$. This served as the motivation for the first three authors to investigate the inner structure of $\CQ_2$, see \cites{ACR,ACR2}. Shortly after the first version of \cite{ACR} was being circulated, we realized that a good part of the questions, answers and techniques entering the proofs in \cite{ACR} have analogues in the much broader setting of boundary quotients of right LCM semigroups. In hindsight, it is fair to say that most of those results and proofs have in fact become clearer and more conceptual as a benefit of the higher level of abstraction.

Let us now describe the $C^*$-algebras to which we extend the line of research started in \cite{ACR}: We shall focus on universal $C^*$-algebras associated to particularly well-behaved examples of right LCM semigroups, that is, left cancellative monoids in which the intersection of any two principal right ideals is either empty, or another principal right ideal again. There are two main types of examples that we have in mind here: 
\begin{enumerate}[(a)]
\item integral dynamics $\Z \rtimes P \subset \Z\rtimes \N^\times$, where $P$ is generated by a family of mutually coprime positive integers, as in \cite{BOS1}, see Example~\ref{ex:integral dynamics}; 
\item semidirect products $G \rtimes_\theta \N$ for an injective endomorphism $\theta$ of a discrete, abelian group $G$ with finite cokernel as appearing in \cite{CuntzVershik}, see Example~\ref{ex:abelian group - single finite endo}.
\end{enumerate}
These are elementary examples of right LCM semigroups of the form $\gxp$ built from \emph{algebraic dynamical systems} $(\gpt)$ in the sense of \cite{BLS2}, that is, $G$ is a countable discrete group, $P$ is a right LCM semigroup, and $\theta\colon P \curvearrowright G$ is an action by injective group endomorphisms such that $pP \cap qP=rP$ implies $\theta_p(G) \cap \theta_q(G) = \theta_r(G)$ for all $p,q \in P$. Let us also recall that $(\gpt)$ is an algebraic dynamical system \emph{of finite type} if the index $[G:\theta_p(G)]$ is finite for all $p \in P$. Some results will be proven for more general algebraic dynamical systems $(\gpt)$ than the ones specified in (a) and (b), see for instance Section~\ref{sec:rel com} and Theorem~\ref{thm:quasi-free-max-1}. However, we assume that the right LCM semigroup $P$ is directed with respect to $p\geq q \stackrel{def}{\Leftrightarrow} p \in qP$. In this case, the boundary quotient $\CQ(\gxp)$, that is, the $C^*$-algebras we intend to study, is the quotient of the full semigroup $C^*$-algebra $C^*(\gxp)$ from \cite{Li1} by the relations $\sum_{[g] \in G/\theta_p(G)} v_{(g,p)}^{\phantom{*}}v_{(g,p)}^* = 1$ for all $p \in P$ for which $\theta_p(G)$ has finite index in $G$, where the $v_{(g,p)}$ denote the standard generating isometries in $C^*(\gxp)$, see \cite{BS1}*{Proposition~4.1}.

By virtue of \cite{BS1}*{Proposition~4.3}, directedness of $P$ is also crucial to have access to a natural representation $\pi\colon \CQ(\gxp) \to \CL(\ell^2(G))$, to which we shall refer as the \emph{canonical representation}. This canonical representation is faithful under moderate assumptions and we start by addressing its irreducibility in Proposition~\ref{prop:can rep is irred}. 
In addition Proposition~\ref{prop:decomp of rep for A_S induced from can rep of Q} gives the decomposition of the induced representation of the torsion subalgebra $\CA_S$ for integral dynamics into irreducible subrepresentations. 

In light of the renewed interest in the study of maximal abelian subalgebras and Cartan subalgebras, see for instance \cites{BL1,BL2,LR17}, we take pain to spot mild conditions ensuring that $C^*(G)$ is a maximal abelian subalgebra of $\CQ(\gxp)$, see Theorem~\ref{thm:C*(G) masa}, and that the diagonal $\CD$ generated by the range projections of the generating isometries in $\CQ(\gxp)$ is a Cartan subalgebra, see Theorem~\ref{thm:D in Q Cartan}. The proofs we give rely on the canonical representation and conditional expectations in both cases. In the case of $C^*(G)$, we observe a rigidity phenomenon for conditional expectations from $\CL(\ell^2(G))$ onto the group von Neumann algebra $W^*(G)$, see Proposition~\ref{prop:cond-expC(G)}, which leads to a uniqueness result for conditional expectations $\CQ(\gxp) \to C^*(G)$, see Remark~\ref{rem:unique cond exp onto C*(G)}.

In Section~\ref{sec:rel com}, we consider general algebraic dynamical systems $(\gpt)$ with $P$ abelian and a strong form of minimality, namely $\bigcap_{p \in P} \theta_p(G) = \{1_G\}$. In this setting, we establish a generalization  of the classical Fourier-coefficient technique from \cite{Cun}, see Lemma~\ref{lem:Cuntzlemma}. This is a key ingredient  to prove Theorem~\ref{thm:rel commutant}, which asserts that the relative commutant of the generating isometries for $P$ inside $\CQ(\gxp)$ is as small as possible, namely $C^*(P^*)$, where $P^*$ is the subgroup of invertible elements in $P$. We remark that we need to assume $P^*$ to be finite for Theorem~\ref{thm:rel commutant} for technical reasons, but this seems likely to be unnecessary.

In contrast to the generality of Section~\ref{sec:rel com}, we focus on the intersection of the two classes described in Example~\ref{ex:integral dynamics} and Example~\ref{ex:abelian group - single finite endo} in Section~\ref{sec:Bogolubov}. That is to say, we study the $p$-adic ring $C^*$-algebras $\CQ_p = \CQ(\Z \rtimes \langle p\rangle)$ for $2\leq p < \infty$. The torsion subalgebra of $\CQ_p$ is then the Cuntz algebra $\CO_p$ generated by $(u^ms_p)_{0 \leq m \leq p-1}$. Understanding the way $\CO_p$ sits inside $\CQ_p$ is a very natural and rewarding task. For example, it was shown in \cite{BOS1} that the inclusion induces a split-injection onto the torsion part in K-theory. Here, we show that the representation $\sigma \colon \CO_p \to \CL(L^2([0,1]))$ known as the \emph{interval picture} extends in a unique way to representation $\widetilde{\sigma} \colon \CQ_p \to \CL(L^2([0,1]))$, which is not unitarily equivalent to the canonical representation, see Proposition~\ref{prop:interval rep from O_p to Q_p}. 

More importantly, we use this result and a fact about the canonical representation to completely solve the problem of extending Bogolubov automorphisms of $\CO_p$ to an endomorphism of $\CQ_p$, see Theorem~\ref{thm:extensible Bogolubov autos}: A Bogolubov automorphism is extendible if and only if it is a gauge automorphism, the exchange automorphism, or a composition of the former two. Moreover, a Bogolubov automorphism admits at most one extension, which either fixes the unitary $u$ (gauge automorphism) or sends it to $u^*$ (exchange automorphism is present). In particular, every extension of a Bogolubov automorphism is an automorphism of $\CQ_p$. This generalization of \cite{ACR}*{Theorem~4.14}  hints at a remarkable rigidity for extensions of certain automorphism groups of the torsion subalgebra to endomorphisms of the boundary quotient $\CQ(\gxp)$. Towards this goal, the contribution of Section~\ref{sec:Bogolubov} is to provide a refined version of the argument given in \cite{ACR} that may pave the way to similar results, for instance in the context of Zappa-Sz{\'e}p products associated to self-similar group actions, see \cite{BRRW}.

In Section~\ref{sec:preserving C*(G)}, we focus on algebraic dynamical systems $(\gpt)$ with abelian group $G$. In this case, the boundary quotient $\CQ(\gxp)$ contains a copy of the group $C^*$-algebra $C^*(G)$, and we investigate the structure of automorphisms of $\CQ(\gxp)$ that preserve $C^*(G)$ globally. To describe the subgroup of all automorphisms that fix $C^*(G)$ pointwise, which we shall denote by ${\rm Aut}_{C^*(G)} \CQ(\gxp)$, we introduce the notion of a $\theta$-twisted homomorphism $\psi\colon P \to C(\widehat{G},\IT)$, where $\widehat{G}$ denotes the Pontryagin dual of $G$, see Definition~\ref{def:twisted Hom(P,U(C*(G)))}. We denote the group formed under pointwise multiplication by ${\rm Hom}_\theta(P,C(\widehat{G},\IT))$. In Theorem~\ref{thm:aut of Q fixing C*(G)}, we then show that ${\rm Hom}_\theta(P,C(\widehat{G},\IT))$ embeds into ${\rm Aut}_{C^*(G)} \CQ(\gxp)$. Under the additional assumption that $C^*(G) \subset \CQ(\gxp)$ is maximal abelian, this embedding is also surjective and ${\rm End}_{C^*(G)} \CQ(\gxp) = {\rm Aut}_{C^*(G)} \CQ(\gxp)$. In particular, we deduce that every element in ${\rm End}_{C^*(G)} \CQ(\gxp)$ is characterized by a family of unitaries in $C^*(G)$, corresponding to a $\theta$-twisted homomorphism $\psi$. 

As a first application of this result, we prove that ${\rm Aut}_{C^*(G)} \CQ(\gxp)$ is a maximal abelian subgroup of ${\rm Aut} \CQ(\gxp)$ in Theorem~\ref{thm:quasi-free-max-1}, assuming that $(\gpt)$ is an algebraic dynamical system of finite type such that $G$ and $P$ are abelian, $C^*(G) \subset \CQ(\gxp)$ is maximal abelian, and $\CQ(\gxp)$ is simple. This result also applies to our motivating examples (a) and (b), see Corollary~\ref{cor:quasi-max-free}. A second application of the picture established in Theorem~\ref{thm:aut of Q fixing C*(G)} is given by Theorem~\ref{thm:aut of Q_N fixing O_n}: For every integral dynamics $(\Z,P,\theta)$, every automorphism of $\CQ(\Z\rtimes P)$ that fixes a Cuntz subalgebra $\CO_n = C^*(\{u^ks_n \mid 0\leq k \leq n-1\})$ pointwise for some $n \geq 2, n \in P$, necessarily belongs to ${\rm Aut}_{C^*(\Z)} \CQ(\Z \rtimes P)$. In addition, we determine the subgroup of ${\rm Hom}_\theta(P,C(\IT,\IT))$ corresponding to ${\rm Aut}_{\CO_n} \CQ(\Z \rtimes P)$. One important consequence of these findings is that two automorphism of $\CQ(\Z \rtimes P)$ are equal if and only if they agree on the torsion subalgebra $\CA_S$, see Corollary~\ref{cor:aut fixing torsion subalgebra}. A heuristic explanation for this phenomenon comes from Theorem~\ref{thm:aut of Q_N fixing O_n} in combination with the crossed product descriptions $\CQ(\gxp) \cong \varinjlim M_p(C^*(\Z)) \rtimes P$ and  $\CA_S \cong \varinjlim M_p(\C) \rtimes P$, see \cite{BOS1} for details.

Within the final Subsection~\ref{subsec:outerness}, we show that many of the automorphisms we considered previously yield outer actions: In the case of (b), every automorphism of $\CQ(G\rtimes_\theta \N)$ that fixes $C^*(G)$, and hence is given by some $f\in C(\widehat{G},\IT)$ due to Theorem~\ref{thm:aut of Q fixing C*(G)}, is outer if $f(1_{\widehat{G}}) \neq 1$, see Proposition~\ref{prop:outer condition for beta_F}. A criterion for the outerness of the gauge action based on functional equations is given in Theorem~\ref{thm:gaugeQpouter} for the case where $C^*(G) \subset \CQ(\gxp)$ is maximal abelian and $P$ is abelian. This applies readily to integral dynamics, see Corollary~\ref{cor:gaugeQouter integral dynamics}. However, the latter result also follows easily from the outerness result for the gauge action in Corollary~\ref{cor:gaugeouter}, which assumes $P$ to be abelian, but $\bigcap_{p \in P} \theta_p(G) = \{1_G\}$ in place of an assumption on $G$. In Theorem~\ref{thm:aut of Q_N inverting u are outer}, we then prove that, for integral dynamics, every automorphism that inverts the generating unitary $u$ is outer. Combining this with Corollary~\ref{cor:gaugeQouter integral dynamics} and Theorem~\ref{thm:extensible Bogolubov autos} shows that the extensions of Bogolubov automorphisms of $\CO_p$ yield an outer action $\IT\times \Z/2\Z \curvearrowright \CQ(\Z\rtimes P)$.\vspace*{3mm}

\emph{Acknowledgements:} This work was initiated during a visit of the fourth author to the first three authors in Rome in June 2016. He thanks his collaborators for quite an enjoyable and productive week, and Sapienza University for its hospitality and financial support. The fourth author was supported by RCN through FRIPRO 240362.

\section{Preliminaries and Notations}\label{sec:prelim}
In what follows we shall be dealing with $C^*$-algebras associated to semigroups built from algebraic dynamical systems. All our semigroups will have an identity, and hence are monoids. By an \emph{algebraic dynamical system} we mean a triple $(G, P,\theta)$, where
\begin{enumerate}[(a)]
\item $G$ is a countable discrete group;
\item $P$ is  a \emph{right LCM} semigroup, that is, a countable left cancellative monoid in which the intersection of two principal right ideals is either empty or another principal right ideal; and
\item $\theta$ is an action of $P$ upon $G$ through injective homomorphisms that  \emph{respects the order}, that is, $pP \cap qP = rP$ implies $\theta_p(G) \cap \theta_q(G) = \theta_r(G)$.
\end{enumerate}
Basics on algebraic dynamical systems $(\gpt)$ are to be found in \cites{BLS1,BLS2}. It is known that the last two conditions are equivalent to the right LCM condition for $\gxp$, given that $G$ is a group. This semidirect product is the actual right LCM semigroup we are interested in.

For the dynamical system $(\gpt)$, the (full) semigroup $C^*$-algebra $C^*(\gxp)$ in the sense of Li provides a natural object to study, see \cite{BLS2}. This $C^*$-algebra admits a canonical representation on $\ell^2(\gxp)$, which is faithful in a number of cases, e.g.~when $G$ is amenable and $P$ is left Ore with amenable enveloping group $P^{-1}P$. However, $C^*(\gxp)$ rather resembles $C^*$-algebras of Toeplitz type. In fact, it can be viewed as a Nica-Toeplitz algebra for a discrete product system of Hilbert bimodules over $P$, see \cite{BLS2}*{Theorem~7.9}.

Inspired by the approach of Crisp and Laca for right-angled Artin groups, see \cite{CrispLaca}, a \emph{boundary quotient} $\CQ(S)$ was defined in \cite{BRRW} for general right LCM semigroups $S$. This quotient of $C^*(S)$ was then studied in connection with structure results for $C^*$-algebras associated to tight groupoids of inverse semigroups, see \cite{Star} and the references therein for details. In \cite{BS1}, the boundary relation for $\CQ(S)$ was analyzed in order to identify this quotient with a previously known $C^*$-algebra in important cases. If $P$ is directed with respect to reverse inclusion of the associated principal right ideals, that is, $p \geq q \Leftrightarrow p \in qP$, then $\CQ(\gxp)$ is the universal $C^*$-algebra generated by a unitary representation $u$ of the group $G$ and a representation $s$ of the monoid $P$ by isometries satisfying the relations
\begin{enumerate}[(I)]
\item $s_pu_g = u_{\theta_p(g)}s_p$,
\item $s_p^*u_g^{\phantom{*}}s_q = \begin{cases} u_{g_1}^{\phantom{*}}s_{p'}^{\phantom{*}}s_{q'}^*u_{g_2} &,\text{ if } g = \theta_p(g_1)\theta_q(g_2) \text{ and } pP \cap qP = pp'P, pp'=qq' \\ 0 &,\text {otherwise.} \end{cases}$
\item $\sum_{\overline{g} \in G/\theta_p(G)} e_{g,p} = 1 \quad \text{if } N_p < \infty$, 
\end{enumerate}
where $e_{g,p} \doteq u_g^{\phantom{*}}s_p^{\phantom{*}}s_p^*u_g^*$ and $N_p \doteq [G:\theta_p(G)]$, see \cite{BS1}*{Proposition~4.1}. We remark that $e_{g,p}$ is independent of the choice of $g \in \overline{g}=g\theta_p(G)$ by (I). 

This result allows for an identification of $\CQ(\gxp)$ with the $C^*$-algebra $\CO[\gpt]$ from \cite{Sta1} for \emph{irreversible algebraic dynamical systems} $(\gpt)$ in the sense of \cite{Sta1}, see \cite{BS1}*{Corollary~4.2}. The latter $C^*$-algebra was constructed as a universal model for the natural realization of the dynamics on $\ell^2(G)$. More generally, it is shown in \cite{BS1}*{Proposition~4.3} that for a general algebraic dynamical system $(\gpt)$, the $C^*$-algebra $\CQ(\gxp)$ admits a canonical representation $\pi\colon \CQ(\gxp) \to \CL(\ell^2(G)), u_gs_p \mapsto U_gS_p$ given by $U_gS_p\xi_h = \xi_{g\theta_p(h)}$ if and only if  $P$ is directed. In addition, we remark that $\pi$ is known to be faithful under moderate assumptions, which in fact guarantee that $\CQ(\gxp)$ is even simple, see \cite{BS1}*{Theorem~4.17}.

Instead of continuing with this increasingly involved discussion of the structure for general algebraic dynamical systems, we have chosen to focus on two natural generalizations of the dynamics considered in \cite{ACR}. There, the authors studied the case of $(\gpt) = (\Z,\N,2)$, i.e.~multiplication by $2$ on the integers. We will consider the following: 

\begin{example}\label{ex:integral dynamics}
Let $S\subset \N^\times \setminus\{1\}$ be a family of relatively prime natural numbers, and $P \subset \N^\times$ the free abelian monoid generated by $S$, i.e. $P = \langle S \rangle$. Then $P$ is right LCM, and acts on $\Z$ by multiplication $\theta_p(n) = pn$ for $p \in P, n \in \Z$. This defines an action $\theta$ that respects the order if and only if $S$ consists of relatively prime numbers. Thus we obtain an (irreversible) algebraic dynamical system $(\Z,P,\theta)$, and the resulting boundary quotient $\CQ(\Z \rtimes_\theta P)$ has the following features:
\begin{enumerate}[(i)]
\item $s_p^*s_q^{\phantom{*}}=s_q^{\phantom{*}}s_p^*$ for all $p,q \in S, p \neq q$.
\item $\CQ(\Z \rtimes_\theta P)$ is the closed linear span of $\{ u^ms_p^{\phantom{*}}s_q^*u^{-n} \mid m,n \in \Z, p,q \in P \}$, see \cite{Sta1}*{Lemma~3.4}.
\item $\CQ(\Z \rtimes_\theta P)$ is a unital UCT Kirchberg algebra, see \cite{Sta1}*{Example~3.29(a)}.
\item The canonical representation $\pi\colon \CQ(\Z \rtimes_\theta P) \to \CL(\ell^2(\Z))$ is faithful.
\item The unitary $u$ in $\CQ(\Z \rtimes_\theta P)$ generates a copy of $C^*(\Z)$.
\item The \emph{diagonal subalgebra} $\CD_S$ of $\CQ(\Z \rtimes_\theta P)$ is generated by the family of commuting projections $\{ u^ns_p^{\phantom{*}}s_p^*u^{-n} \mid (n,p) \in \Z\rtimes_\theta P\}$. Its spectrum is given by the $S$-adic completion of $\Z$.
\item There is a natural gauge action $\gamma$ of the $\lvert S \rvert$-dimensional torus on $\CQ(\Z \rtimes_\theta P)$ given by $\gamma_\chi(u)=u, \gamma(s_p)=\chi(p)s_p$ for $p \in S$. The fixed-point algebra $\CF$ for $\gamma$ is the Bunce-Deddens algebra of type $(\prod_{p \in S}p)^\infty$, see \cite{Sta1}*{Example~3.29(a)}.
\item There is another distinguished subalgebra: 
\[\CA_S \doteq C^*(\{ u^ns_p \mid p \in S, 0 \leq n \leq p-1 \}) \subset \CQ(\Z \rtimes_\theta P).\] 
It is known through \cite{BOS1}*{Corollary~5.2 and Corollary~5.4} that $\CA_S$ is also a unital UCT Kirchberg algebra like $\CQ(\Z \rtimes_\theta P)$, and that the canonical inclusion $\CA_S \hookrightarrow \CQ(\Z \rtimes_\theta P)$ yields a split-injection onto the torsion part of $K_*(\CQ(\Z \rtimes_\theta P))$. For this reason, $\CA_S$ was named the \emph{torsion subalgebra} in \cite{BOS1}. It was also shown that $\CA_S$ is isomorphic to $\bigotimes_{p \in S} \CO_p$ if $\lvert S \rvert \leq 2$ or the greatest common divisor of $S-1 \subset \N^\times$ is $1$, see \cite{BOS1}*{Theorem~6.4}.
\end{enumerate} 
\end{example}

While Example~\ref{ex:integral dynamics} promotes the direction of considering actions of higher dimensional semigroups $P$ on the same group, we can equally well stay with the case of a single endomorphism, and allow the group to be more complicated than $\Z$:

\begin{example}\label{ex:abelian group - single finite endo}
Suppose $G$ is a discrete, abelian group, and $\theta$ is an injective group endomorphism of $G$ with finite cokernel. Then $(G,\N,\theta)$ form an (irreversible) algebraic dynamical system. In addition, we shall assume that $(G,\N,\theta)$ is \emph{minimal} in the sense of \cite{Sta1}, that is, $\bigcap_{n \in \N} \theta^n(G) = \{1_G\}$. This is equivalent to minimality of the dynamical system formed by the Pontryagin dual $\widehat{G}$ of $G$ and the dual endomorphism $\hat{\theta}$, see \cite{Sta2}. The $C^*$-algebra $\CQ(G \rtimes_\theta \N)$ then has the following features:
\begin{enumerate}[(i)]
\item $\CQ(G \rtimes_\theta \N)$ is the closed linear span of $\{ u_gs^ms^{n*}u_h \mid g,h \in G, m,n \in \N\}$, see \cite{Sta1}*{Lemma~3.4}.
\item $\CQ(G \rtimes_\theta \N)$ is a unital UCT Kirchberg algebra if and only if $\bigcap_{n \in \N}\theta^n(G) = \{1_G\}$, see \cite{Sta1}*{Corollary~3.28} and \cite{BS1}*{Corollary~4.14}.
\item The canonical representation $\pi\colon \CQ(G \rtimes_\theta \N) \to \CL(\ell^2(G))$ is faithful.
\item The unitary representation $u$ in $\CQ(G \rtimes_\theta \N)$ generates a copy of $C^*(G)$ as $G$ is abelian, hence amenable.
\item The \emph{diagonal subalgebra} $\CD_\theta$ of $\CQ(G \rtimes_\theta \N)$ is generated by the commuting projections $\{ u_g^{\phantom{*}}s^ns^{n*}u_g^* \mid (g,n) \in G\rtimes_\theta \N\}$. Its spectrum $G_\theta$ is a completion of $G$ with respect to $\theta$. This Cantor space is actually a compact abelian group, see \cite{Sta1}*{Remark~4.1~c)}. 
\item There is a natural gauge action $\gamma$ of the torus on $\CQ(G \rtimes_\theta \N)$ given by $\gamma_z(u_g)=u_g$ for $g \in G$, and $\gamma(s)=zs$. The fixed-point algebra $\CF$ for $\gamma$ is a generalized Bunce-Deddens algebra, see \cite{Sta1}*{Proposition~4.2}.
\item For each \emph{transversal} $\CT$ for $G/\theta(G)$, that is, a minimal complete set of representatives, we can consider the subalgebra $\CA_\CT \doteq C^*(\{ u_gs \mid g \in \CT\})$ of $\CQ(G \rtimes_\theta \N)$. We remark that, unlike the case of $G=\Z$ from Example~\ref{ex:integral dynamics}, there is no canonical choice for $\CT$ in general. On the other hand, $\CA_\CT$ is always isomorphic to $\CO_{N_\theta}$ with $N_\theta \doteq [G:\theta(G)]$ because the generators form a Cuntz family of isometries due to (III).
\end{enumerate}
\end{example}
We observe that the $p$-adic ring $C^*$-algebras form a family of algebras being in the intersection of the two above mentioned cases. In fact, we have that $\CQ_p=\CQ(\IZ\rtimes_{\theta_p} \IN)$ (cf. \cite{BOS1}*{Definition 2.1 and Proposition 2.12}).

The following two propositions about the canonical representation parallel the preliminary results we needed in \cite{ACR}.

\begin{proposition}\label{prop:can rep is irred}
Suppose $(\gpt)$ is an algebraic dynamical system for which $P$ is directed. If $(\gpt)$ is minimal, that is, $\bigcap_{p \in P} \theta_p(G) = \{1_G\}$ holds, then the canonical representation $\pi\colon \CQ(\gxp) \to \CL(\ell^2(G))$ is irreducible.
\end{proposition}
\begin{proof}
Let $(\xi_g)_{g \in G}$ denote the standard orthonormal basis for $\ell^2(G)$, and $M \subset \ell^2(G)$ be a $\pi(\CQ(\gxp))$-invariant closed subspace. Then the orthogonal projection $Q$ onto $M$ belongs to $\pi(\CQ(\gxp))'$. Hence we have $S_pQ\xi_{1_G} = Q\xi_{1_G}$, i.e.~$Q\xi_{1_G}$ is an eigenvector for each $S_p$ with eigenvalue $1$. The condition $\bigcap_{p \in P} \theta_p(G) = \{1_G\}$ now implies that $Q\xi_{1_G} \in \C\xi_{1_G}$, so that either $\xi_{1_G} \in M$ or $\xi_{1_G} \in M^\perp$. In the first case, the invariance of $M$ under $\pi(u_g), g \in G$, yields $\xi_g \in M$ for all $g \in G$, so that $M=\ell^2(G)$. The second case is analogous as $M^\perp$ is necessarily $\pi(\CQ(\gxp))$-invariant as well.
\end{proof}

In the case of Example~\ref{ex:integral dynamics}, we can also describe the decomposition of the induced representation $\pi|_{\CA_S}$ of $\CA_S$ into irreducible subrepresentations. Indeed, the arguments from \cite{ACR}*{Propositon~2.6 -- Corollary~2.9} carry over verbatim, where one has to replace $\CO_2$ by $\CA_S$. Letting $E_+, E_- \in \CL(\ell^2(\Z))$ denote the orthogonal projection onto $\ell^2(\Z_{\geq 0})$ and $\ell^2(\Z_{< 0})$, respectively, we arrive at:

\begin{proposition}\label{prop:decomp of rep for A_S induced from can rep of Q}
For $(\gpt)$ as in Example~\ref{ex:integral dynamics}, the representation of $\CA_S$ obtained as the restriction of the canonical representation $\pi\colon \CQ(\Z\rtimes P) \to \CL(\ell^2(\Z))$ decomposes into two disjoint, irreducible representations $\pi_+ \doteq E_+\pi E_+$ and $\pi_- \doteq E_-\pi E_-$. In particular, we have $\pi(\CA_S)' = \C E_+ \oplus \C E_-$ and $\pi(\CA_S)'' =  \CL(\ell^2(\Z_{\geq 0})) \oplus \CL(\ell^2(\Z_{< 0}))$. 
\end{proposition}

\section{Maximal abelian subalgebras}\label{sec:masas}
\subsection{The group C*-algebra}
In this subsection, we first restrict to the following setup for showing that $C^*(G)$ is maximal abelian in $\CQ(\gxp)$, see Theorem~\ref{thm:C*(G) masa}. Recall that an algebraic dynamical system $(\gpt)$ is said to be of \emph{finite type} if $N_p \doteq [G:\theta_p(G)]$ is finite for all $p \in P$.

\begin{example} \label{ex:C*(G) masa setup}
Suppose $(\gpt)$ is a commutative algebraic dynamical system of finite type with directed $P$, such that the canonical representation $\pi\colon \CQ(\gxp) \to \CL(\ell^2(G))$ is faithful. In addition, assume that for all $p,q \in P, p \neq q$ there exists $g \in G$ of infinite order such that the group endomorphism $\Phi_{p,q}$ of $G$ given by $\Phi_{p,q}(h) \doteq \theta_p(h)\theta_q(h^{-1})$ is injective on $\langle g \rangle \cong \Z$.
\end{example}

The last part of the assumptions in Example~\ref{ex:C*(G) masa setup} holds for instance if $G$ is not a torsion group and $\Phi_{p,q}$ is injective for all distinct $p, q\in P$. All the assumptions in Example~\ref{ex:C*(G) masa setup} are satisfied by all examples described in Example~\ref{ex:integral dynamics} or Example~\ref{ex:abelian group - single finite endo} (due to minimality). From now on, we assume that $(\gpt)$ is as specified in Example~\ref{ex:C*(G) masa setup}.

\begin{lemma}\label{lem:empty point spectrum}
For all distinct $p,q \in P$, there exists $g \in G$ such that the point spectrum of $\pi(u_{\Phi_{p,q}(g)}) \in \CL(\ell^2(G))$ is empty.
\end{lemma}
\begin{proof}
Let $g \in G$ have the property described in Example~\ref{ex:C*(G) masa setup} for given $p \neq q$. Then the element $\Phi_{p,q}(g)= \theta_p(g)\theta_q(g^{-1})$ still has infinite order, as $1_G=(\theta_q(g^{-1}) \theta_p(g))^n=\Phi_{p,q}(g^n)$ forces $g^n=1_G$. The conclusion then follows from the general observation that the point spectrum of $\pi(u_k)$ is empty whenever $k\in G$ is of infinite order. Indeed, suppose there was an eigenvector $\xi = \sum_{h\in G} c_h \xi_h \in\ell^2(G)$ of $\pi(u_k)$, and say $c_{h'} \neq 0$. Then we would get $\sum_{h\in G} c_{k^{-1}h} \xi_h = \pi(u_k) \xi = \lambda \xi =\sum_{h\in G} \lambda c_h \xi_h$ for some $\lambda\in\IT$ (as $u_k$ is a unitary). Thus we would have $|c_{k^{m}h'}| = |c_{h'}| \neq 0$ for all $m \in \Z$. Since $k$ is of infinite order, this contradicts $\xi \in \ell^2(G)$, and we apply this to $\Phi_{p,q}(g)$.
\end{proof}

Denote by $W^*(G) \subset \CL(\ell^2(G))$ the group von Neumann algebra of $G$. Then there exists a conditional expectation $E_1\colon \CL(\ell^2(G)) \to  W^*(G)$, see \cite{KS}. 

\begin{proposition}\label{prop:cond-expC(G)}
If $(\gpt)$ is as described in Example~\ref{ex:C*(G) masa setup}, then every conditional expectation $E\colon \CL(\ell^2(G)) \to  W^*(G)$ satisfies 
\begin{equation}\label{eq:cond exp C*(G)}
E \circ \pi(u_g^{\phantom{*}}s_p^{\phantom{*}}s_q^*u_h^*) =  \delta_{p,q} \ N_p^{-1} u_{gh^{-1}}\quad \text{ for all } g,h \in G, p,q \in P.
\end{equation}
In particular, every conditional expectation $E\colon \CL(\ell^2(G)) \to  W^*(G)$ restricts to the conditional expectation $E_0\colon \pi(\CQ(\gxp)) \to \pi(C^*(G))$ given by
\[E_0\circ \pi(u_g^{\phantom{*}}s_p^{\phantom{*}}s_q^*u_h^*) = \delta_{p,q} \ N_p^{-1} \pi(u_{gh^{-1}})\quad \text{ for all } g,h \in G, p,q \in P.\] 
\end{proposition}
\begin{proof}
Since $G$ is abelian, $W^*(G) = \pi(C^*(G))'$ is abelian so that $E\circ \pi(u_g^{\phantom{*}}s_p^{\phantom{*}}s_q^*u_h^*) = \pi(u_{gh^{-1}}) E\circ \pi(s_p^{\phantom{*}}s_q^*)$ for all $g,h \in G, p,q \in P$. In particular, (III) implies 
\[\begin{array}{c}
1 = E \circ \pi \bigl(\sum\limits_{\overline{g} \in G/\theta_p(G)} u_g^{\phantom{*}}s_p^{\phantom{*}}s_p^*u_g^* \bigr) = N_p \ E\circ\pi(s_p^{\phantom{*}}s_p^*),
\end{array}\]
so that $E \circ \pi(u_g^{\phantom{*}}s_p^{\phantom{*}}s_p^*u_g^*) =  N_p^{-1}$.

On the other hand, for $p\neq q$, (I) yields 
\[ \pi(u_{\theta_q(g)}^{\phantom{*}})E(\pi(s_p^{\phantom{*}} s_q^*)) = E(\pi(s_p^{\phantom{*}} s_q^*u_{\theta_q(g)}^{\phantom{*}})) = E(\pi(u_{\theta_p(g)}^{\phantom{*}} s_p^{\phantom{*}} s_q^*)) = \pi(u_{\theta_p(g)}^{\phantom{*}}) E(\pi(s_p^{\phantom{*}} s_q^*)),\]
which is equivalent to $(1-\pi(u_{\Phi_{p,q}(g)}))E(\pi(s_p^{\phantom{*}} s_q^*))=0$.
By Lemma~\ref{lem:empty point spectrum}, the value $1$ is not an eigenvalue of $\pi(u_{\Phi_{p,q}(g)}) \in \CL(\ell^2(G))$ for $g \in G$ chosen according to the hypothesis in Example~\ref{ex:C*(G) masa setup} for $p \neq q$. Thus $1-\pi(u_{\Phi_{p,q}(g)})$ is injective, which allows us to conclude that $E\circ\pi(s_p^{\phantom{*}} s_q^*) = 0$. 
\end{proof}

\begin{remark}\label{rem:unique cond exp onto C*(G)}
Due to the faithfulness of the canonical representation, the proof of Proposition~\ref{prop:cond-expC(G)} in fact shows that $ \pi^{-1} \circ E_0 \circ \pi \colon\CQ(\gxp) \to C^*(G)$ is the unique conditional expectation from $\CQ(\gxp)$ onto $C^*(G)$.
\end{remark}

\begin{thm}\label{thm:C*(G) masa}
Suppose $(\gpt)$ satisfies the conditions in Example~\ref{ex:C*(G) masa setup}. If $(\gpt)$ is minimal, then the $C^*$-algebra $C^*(G)$ is maximal abelian in $\CQ(\gxp)$.
\end{thm}
\begin{proof}
It will be convenient to work in the representation $\pi\colon \CQ(\gxp) \to \CL(\ell^2(G))$. Let $x\in C^*(G)'\cap \CQ(\gxp)$ and choose a sequence $(x_k)_{k \in \N} \subset {\rm span}\{ u_g^{\phantom{*}}s_p^{\phantom{*}}s_q^*u_h \mid g,h \in G, p,q \in P\}$ with $\|x-x_k\| \to 0$. Then Proposition~\ref{prop:cond-expC(G)} entails
\[\|\pi(x) - E_0(\pi(x_k))\| = \|E_1(\pi(x)) - E_0(\pi(x_k))\| = \|E_0(\pi(x-x_k))\| \leq \|x-x_k\| \to 0\]
because $\pi(x) \in \pi(C^*(G))' = W^*(G)$ and $\|E_0\circ \pi\| \leq 1$. Since $E_0\circ \pi$ expects onto $\pi(C^*(G))$ and $\pi$ is faithful, we conclude that $x \in C^*(G)$.
\end{proof}

\begin{remark}\label{rem:masa for P free abelian monoid}
The canonical representation is clearly faithful if $\CQ(\gxp)$ is simple, which is closely linked to the condition $\bigcap_{p \in P} \theta_p(G) = \{1_G\}$ in the situation of Example~\ref{ex:C*(G) masa setup}, see for instance \cite{BS1}*{Remark~4.15}. In the case where $P$ is free abelian, minimality of $(\gpt)$ was already known to imply that $C^*(G) \subset \CQ(\gxp)$ is a maximal abelian subalgebra, see \cite{Sta2}*{Corollary~2.7, Theorem~5.6 and Proposition~6.1}. Moreover, minimality of $(\gpt)$ is equivalent to simplicity of $\CQ(\gxp)$ in this situation, see \cite{Sta2}*{Corollary~6.2}.
\end{remark}

It is certainly an interesting task to investigate under what conditions  $C^*(G)$ is not only maximal abelian, see Theorem~\ref{thm:C*(G) masa}, but in fact a Cartan subalgebra of $\CQ(\gxp)$. In the case of the diagonal $\CD$, this question has a clear and affirmative answer, as we shall see in the next subsection.

\subsection{The diagonal}
Let us now turn our attention to the diagonal subalgebra $\CD$ of $\CQ(\gxp)$, that is, the commutative subalgebra of $\CQ(\gxp)$ generated by the projections $e_{g,p} \doteq u_g^{\phantom{*}}s_p^{\phantom{*}}s_p^* u_g^*$ with $(g,p) \in \gxp$, see \cite{Sta1}*{Lemma~3.5}. We note that
\[\CD= \overline{{\rm span}} \{ u_g^{\phantom{*}}s_p^{\phantom{*}}s_p^*u_{g^{-1}} \mid (g,p) \in \gxp \}.\]
In many cases, there exists a faithful conditional expectation $\Theta\colon \CQ(\gxp) \to \CD$ given by $u_g^{\phantom{*}}s_p^{\phantom{*}}s_q^* u_h^* \mapsto \delta_{g,h}\delta_{p,q} e_{g,p}$, for instance if $G$ and $P$ are abelian so that we obtain $\Theta$ as the composition of two faithful conditional expectations obtained from averaging over the gauge actions of the duals of the Grothendieck group of $P$ and of the dual group of $G$ using $\CQ(\gxp) \cong (\CD \rtimes G) \rtimes P$, see \cite{Sta1}*{Corollary 3.22}. In order to recover the maximality of $\CD$ in this general setting we need to make an extra assumption:
\begin{equation}\label{eq:only trivial primage match}
\text{For all } p,q \in P, p \neq q, \text{ every $g \in G$ has finitely many preimages under } \Phi_{p,q},
\end{equation}
where $\Phi_{p,q}(h) = \theta_p(h)\theta_q(h)^{-1}$, see Example~\ref{ex:C*(G) masa setup}. Equation~\ref{eq:only trivial primage match} holds for example for the integral dynamics $(\Z,P,\theta)$ described in Example~\ref{ex:integral dynamics}: given $m \in \Z$ and $p,q \in \N^\times$ with $p>q$, the only solution to $m=pn -qn$ is of course $n= (p-q)^{-1}m$ (which belongs to $\Z$ in case $m \in (p-q)\Z$). Thus the sets are not only finite, but singletons. However, \eqref{eq:only trivial primage match} does not hold as soon as there is $p \in P\setminus \{1_P\}$ for which the endomorphism $\theta_p$ fixes some group element $h \neq 1_G$ of infinite order. Indeed, we then get $pp \neq p$ by left cancellation for $P$, but $1_G = \theta_{pp}(h)\theta_p(h)^{-1}$. For instance, this happens for $(\Z^2,\theta,\N)$ given by the diagonal matrix $\theta_1= {\rm diag}(p,1)$ with $p \in \Z\setminus \{-1,0,1\}$.

\begin{lemma}\label{lem:cond exp onto D compared with the one onto l^infty}
Let $(\gpt)$ be an algebraic dynamical system with directed $P$, and let $\CE$ be the conditional expectation from $\CL(\ell^2(G))$ onto $\ell^\infty(G)$ given by $\CE(V)\xi_g \doteq \langle V\xi_g, \xi_g\rangle \xi_g$. If \eqref{eq:only trivial primage match} holds, then the restriction of $\CE|_{\pi(\CQ(\gxp))}$ coincides with $\pi \circ \Theta$ up to compact operators. 
\end{lemma}
\begin{proof}
We observe that for $g,h,k \in G$ and $p,q \in P$,
\[\CE(\pi(u_g^{\phantom{*}}s_p^{\phantom{*}}s_q^*u_h^*))\xi_k =\begin{cases}
\xi_k &, \text{if } h^{-1}k \in \theta_q(G) \text{ and } k= g\theta_p(\theta_q^{-1}(h^{-1}k)),\\
0 &, \text{otherwise.} \end{cases}\]
In the case of $p=q$, the second part in the first condition amounts to $g=h$ so that the operator needs to be diagonal. This is exactly what happens for $\Theta$. For $p \neq q$, assume $k=h\theta_q(\ell)$ for some $\ell \in G$. Then the second part turns into $g^{-1}h = \theta_p(\ell)\theta_q(\ell)^{-1}$. By \eqref{eq:only trivial primage match}, there are only finitely many $\ell \in G$ that satisfy this, so $\CE(\pi(u_g^{\phantom{*}}s_p^{\phantom{*}}s_q^*u_h^*))$ and $\pi(\Theta(u_g^{\phantom{*}}s_p^{\phantom{*}}s_q^*u_h^*))$ differ by a finite rank operator. Hence the claim follows from $\CQ(\gxp) = \overline{{\rm span}}\{ u_g^{\phantom{*}}s_p^{\phantom{*}}s_q^*u_h^* \mid (g,p),(h,q) \in \gxp\}$.
\end{proof}

\begin{remark}\label{rem:cond exp onto D intertwined by can rep}
In particular, if in the situation of Lemma~\ref{lem:cond exp onto D compared with the one onto l^infty} the algebra $\pi(\CQ(\gxp))$ does not contain nontrivial compacts, then $\CE(\pi(x)) = \pi(\Theta(x))$ for every $x \in \CQ(\gxp)$ with $\CE(\pi(x)) \in \pi(\CQ(\gxp))$. This is for instance true if $\CQ(\gxp)$ is purely infinite and simple.
\end{remark}

\begin{lemma}\label{lem:D'' equal to l^infty(G)}
If $(\gpt)$ is a minimal algebraic dynamical system with directed $P$, then $\pi(\CD)'=\pi(\CD)''=\ell^\infty(G)$.
\end{lemma}
\begin{proof}
By the minimality assumption, the net of projections $(\pi(u_g^{\phantom{*}}s_p^{\phantom{*}}s_p^*u_g^*))_{p \in P} \subset \CD$ converges strongly to the projection onto $\IC\xi_g$ for every $g \in G$. Since $\pi(\CD)''$ is a von Neumann algebra and hence strongly closed, we deduce that all these minimal projections belong to $\pi(\CD)''$, and hence $\pi(\CD)''=\ell^\infty(G)$. Now $\pi(\CD)' = \pi(\CD)''' = \ell^\infty(G)'=\ell^\infty(G)$ proves the remaining assertion.
\end{proof}

\begin{thm}\label{thm:D in Q Cartan}
Suppose $(\gpt)$ is a minimal algebraic dynamical system with directed $P$ such that $\Theta\colon \CQ(\gxp) \to \CD$ is a faithful conditional expectation and \eqref{eq:only trivial primage match} holds. Then the diagonal $\CD$ is a Cartan subalgebra of $\CQ(\gxp)$.
\end{thm}
\begin{proof}
By assumption, $\Theta$ is a faithful conditional expectation. Moreover, this allows us to deduce simplicity and pure infiniteness for $\CQ(\gxp)$ out of minimality, that is, $\bigcap_{(g,p) \in \gxp} g\theta_p(G)g^{-1} = \{1_G\}$, very much in the way \cite{Sta1}*{Theorem~3.26} is proven. In particular, the canonical representation $\pi$ is faithful. To show that $\CD$ is maximal abelian in $\CQ(\gxp)$, we follow the strategy of \cite{ACR}*{Theorem~3.9}: Suppose $x \in \CD' \cap \CQ(\gxp)$. By Lemma~\ref{lem:D'' equal to l^infty(G)}, the operator $\pi(x)$ belongs to $\ell^\infty(G) \cap \pi(\CQ(\gxp))$, so that Remark~\ref{rem:cond exp onto D intertwined by can rep} implies that $\pi(x) = \CE(\pi(x)) = \pi(\Theta(x))$. Finally, the normalizer of $\CD$  in $\CQ(\gxp)$ generates $\CQ(\gxp)$ since
one has, for all $h,g \in G$, $p,q \in P$, 
\begin{align*}
u_h^{\phantom{*}}(u_g^{\phantom{*}} s_p^{\phantom{*}} s_p^* u_g^*)u_h^* & =u_{gh}^{\phantom{*}} s_p^{\phantom{*}} s_p^* u_{gh}^* , \quad
%
s_q(u_g s_p s_p^* u_g^*)s_q^* = u_{\theta_q(g)} s_{pq} s_{pq}^* u_{\theta_q(g)}^* \\
s_q^*(u_g s_p s_p^* u_g^*)s_q  & = \begin{cases} u_{g_1}^{\phantom{*}}s_{q'}^{\phantom{*}}s_{p'}^*u_{g_2}(u_{g_1}^{\phantom{*}}s_{q'}^{\phantom{*}}s_{p'}^*u_{g_2})^*  \\ 0  \end{cases}\\
& = \begin{cases} u_{g_1}^{\phantom{*}}s_{q'}^{\phantom{*}}s_{q'}^*u_{g_1}^* \in \CD &\text{if } g = \theta_q(g_1)\theta_p(g_2), 
 pP \cap qP = pp'P, pp'=qq' \\ 0 & \text {otherwise} \end{cases}
\end{align*}
showing that $u_h,s_q\in \CN_{\CD}(\CQ(\gxp))$. 
\end{proof}

For the sake of completeness we would like to observe that an elegant, but less elementary,
proof of a strengthening of the above result can be achieved as follows, compare \cite{ACR}*{end of Section~3.1}: Under the assumptions that $G$ is amenable and $P$ embeds into an amenable group $H$ (for simplicity take $H = \langle P\rangle$), the algebra $\CQ(\gxp)$ is a reduced partial transformation groupoid $C^*$-algebra. If $\gxp$ happens to be minimal, then the corresponding partial action of $G \rtimes H$ is topologically free. The diagonal $\CD$ is then seen to be a Cartan subalgebra in $\CQ(\gxp)$ as an application of \cite{Ren}*{Proposition~3.1}. In addition, let us mention that \cite{LR17} applies to our situation, showing that there are infinitely many pairwise inequivalent Cartan subalgebras of $\CQ(\gxp)$, because the latter will be a UCT Kirchberg algebra in many of the above cases.

In hindsight, we thus see that \eqref{eq:only trivial primage match} is an artifact of our strategy of proof here, but we still find it interesting that a variant of it enters the stage for both commutative subalgebras $\CD$ and $C^*(G)$, see Example~\ref{ex:C*(G) masa setup}.

\section{The relative commutant of a family of generating isometries}\label{sec:rel com}
Throughout this section we shall assume that $(\gpt)$ is an algebraic dynamical system with $P$ abelian and $\bigcap_{p \in P} \theta_p(G) = \{1_G\}$. Requiring $P$ to be abelian grants us access to the Grothendieck group $P^{-1}P$ of $P$, which is a discrete abelian group. We denote by $T$ its Pontryagin dual, which is a compact abelian group. The group $T$ acts on $\CQ(\gxp)$ by a gauge action $\gamma$ with $\gamma_\chi(u_gs_p) = \chi(p) \ u_gs_p$, and we denote by $\CF$ the fixed-point algebra for $\gamma$, see Example~\ref{ex:integral dynamics}~(vii) and Example~\ref{ex:abelian group - single finite endo}~(vi). We remark that $\CF = \overline{{\rm span}}\{ u_g^{\phantom{*}}s_p^{\phantom{*}}s_p^*u_h^* \mid g,h \in G, p \in P\}$. 

The aim of this section is to show that  $C^*(\{s_p \mid p \in P\})'\cap \CQ(\gxp)$ is as small as possible, see Theorem~\ref{thm:rel commutant}. 

If $P$ is abelian, it is in particular directed with respect to $p \geq q \Leftrightarrow p \in qP$, so there is a sequence $(p_n)_{n \in \N} \subset P$ with $p_{n+1} \in p_nP$ for all $n$ such that for every $p \in P$ we have $p_n \in pP$ for $n$ large enough. 
We shall need the following result in Lemma~\ref{lem:unitaries in rel commutant}.

\begin{lemma}\label{lem:limfixpoint}
For $x \in \CF$, the sequence $(s_{p_m}^*xs_{p_m}^{\phantom{*}})_{m \in \N}$ converges to a limit in $\C$.
\end{lemma}
\begin{proof}
First, let $x \in {\rm span}\{ u_g^{\phantom{*}}s_p^{\phantom{*}}s_p^*u_h \mid g,h \in G, p \in P\}$. As $x$ is a finite linear combination of elements of the form $u_g^{\phantom{*}}s_p^{\phantom{*}}s_p^*u_h$, there is $m_0 \in \N$ such that $p_m \in pP$ for all $p$ that appear and all $m\geq m_0$. Therefore, we get $s_{p_m}^*xs_{p_m}^{\phantom{*}} \in \IC [G] \subset \CQ(\gxp)$ for all $m\geq m_0$ by (II), more precisely $s_{p_m}^*xs_{p_m}^{\phantom{*}} = \sum_{g\in F} d_g u_g$ with $d_g \in \C$ for a finite set $F \subset G$. As $\bigcap_{p \in P}\theta_p(G) = \{1_G\}$, we can choose $p \in P$ such that $g \notin \theta_p(G)$ for all $g \in F\setminus\{1_G\}$. Let $m'\in \N$ be large enough so that $p_m \in p_{m_0}pP$ for all $m \geq m'$. For such $m$, we then have $s_{p_m}^*xs_{p_m}^{\phantom{*}} = d_1 \doteq c_x \in \C$, using (II).

Now let $x \in \CF$ and pick a sequence $(x_k)_{k \in \N}\subset {\rm span}\{ u_g^{\phantom{*}}s_p^{\phantom{*}}s_p^*u_h \mid g,h \in G, p \in P\}$ with $x_k \to x$. Let $m_k \in \N$ be an $m'$ from the first part for $x_k$, and define a net $(x_{k,m})_{(k,m) \in \Lambda}$ by $\Lambda \doteq \{ (k,m) \in \N^2 \mid m \geq m_k\}$ and $x_{k,m} \doteq s_{p_m}^*x_ks_{p_m}^{\phantom{*}}$. Then $(x_{k,m})_{(k,m) \in \Lambda}$ is a Cauchy net (in $\C$) as 
\[\|x_{k,m} - x_{\ell,n}\| = |c_{x_k}-c_{x_\ell}| = \|x_{k,m+n} - x_{\ell,m+n}\| = \|s_{p_{m+n}}^*(x_k-x_\ell)s_{p_{m+n}}^{\phantom{*}}\| \leq \|x_k-x_\ell\|\]
and $x_k \to x$. Therefore, $x_{k,m} \to c$ for some $c \in \C$ because $\C$ is complete. From this we easily deduce $s_{p_m}^*xs_{p_m}^{\phantom{*}} \to c \in \C$ as $\|s_{p_m}^*xs_{p_m}^{\phantom{*}} - x_{k,m}\| \leq \|x-x_k\| \to 0$.
\end{proof}

Recall that, given $p,q \in P$, an element $r \in P$ satisfying $pP\cap qP = rP$ is called a right least common multiple (right LCM) of $p$ and $q$. Commutativity in $P$ implies that we do not need to distinguish between right and left multiples, but more importantly, the notion of a \emph{greatest common divisor} (GCD) also makes sense for elements in $P$: For $p,q \in P$, an element $r \in P$ is a \emph{common divisor} if $p,q \in rP$. An element $d \in P$ is said to be the \emph{greatest common divisor} (GCD) for $p$ and $q$, if every common divisor $\tilde{d}$ of $p$ and $q$ satisfies $d \in \tilde{d}P$. For every pair $p$ and $q$, the GCD exists by an application of Zorn's Lemma, and it is unique up to multiplication by the subgroup of invertible elements $P^*$. The GCD relates to the (right) LCM as follows:

\begin{proposition}\label{prop:GCD and LCM for abelian right LCM}
Let $P$ be an abelian right LCM semigroup. If $d$ is a greatest common divisor and $m$ a least common multiple for $p,q \in P$, then $mP^* = d^{-1}pqP^*$.  
\end{proposition}
\begin{proof}
As $P$ is abelian and $d^{-1}p,d^{-1}q \in P$, we have $d^{-1}pq \in pP\cap qP = mP$. Conversely, let $a,b \in P$ with $m=pa=qb$. Since $pq \in mP$, there is $c \in P$ such that $pq = mc = pac = qbc$. By left cancellation, this says that $q=ac$ and $p=bc$, that is, $c$ is a common divisor of $p$ and $q$. Thus we get $d =ce$ for some $e \in P$, and hence $m=pd^{-1}qe \in pd^{-1}qP$. As $P$ is right LCM, this shows $mP^* = d^{-1}pqP^*$.
\end{proof}
If $P = \IN^\times$, the claim in the previous proposition is nothing but the well-known equality $nm = {\rm GCD}(n,m) \cdot {\rm lcm}(n, m)$ for all $n,m \in \IN$.

\medskip
The set $P(\perp)\subset P\times P$ and the maps $F_{(p,q)}\colon \CQ(\gxp) \to \CF$ for $(p,q) \in P(\perp)$ introduced in the following remark will be crucial to the proof of Lemma~\ref{lem:Cuntzlemma}, which is the heart of the whole section and goes all the way back to \cite{Cun}*{Proposition 1.10}.

\begin{remark}\label{rem:Fourier-maps}
Let $P(\perp) \doteq \{ (p,q) \in P\times P\mid pP\cap qP = pqP\}$ denote the collection of all \emph{relatively prime pairs} in $P$. The terminology alludes to the fact that $pP\cap qP = pqP$ is equivalent to the GCD of $p$ and $q$ being in $P^*$, see Proposition~\ref{prop:GCD and LCM for abelian right LCM}. For all $(p,q) \in P(\perp)$, the isometries $s_p$ and $s_q$ in $\CQ(\gxp)$ doubly commute, i.e. 
\begin{equation}\label{eq:(p,q) in P(perp) -> doubly commute}
s_p^*s_q=s_qs_p^*,
\end{equation}
as easily follows by (II), applied to $g=g_1=g_2=1_G$, $p'=q$ and $q'=p$.

Two pairs $(p,q),(\tilde{p},\tilde{q}) \in P(\perp)$ satisfy $p^{-1}q = \tilde{p}^{-1}\tilde{q}$ if and only if there is $x \in P^*$ such that $\tilde{p}=xp$ and $\tilde{q}=xq$. This relation defines an equivalence relation on $P(\perp)\times P(\perp)$ that we denote by $\sim$. For each $(p,q) \in P(\perp)$, we define a contractive, linear map $F_{(p,q)}\colon \CQ(\gxp) \to \CF$ by $a \mapsto \int_T \gamma_\chi(s_p^{\phantom{*}}as_q^*) \ d\chi$, where the integration is with respect to the normalized Haar measure on the compact abelian group $T$. We observe that 
\begin{equation}\label{eq:F_(p,q) range projections}
F_{(p,q)}(a) = s_p^{\phantom{*}}s_p^*F_{(p,q)}(a)s_q^{\phantom{*}}s_q^* \quad \text{for all } a \in \CQ(\gxp), (p,q) \in P(\perp).
\end{equation}
Moreover, we note that
\begin{equation}\label{eq:F_(p,q) and sim}
s_p^*F_{(p,q)}(a)s_q^{\phantom{*}} = s_{\tilde{p}}^*F_{(\tilde{p},\tilde{q})}(a)s_{\tilde{q}}^{\phantom{*}} \quad \text{whenever } (p,q)\sim (\tilde{p},\tilde{q}). 
\end{equation}
Finally, for every $u_g^{\phantom{*}}s_{p'}^{\phantom{*}}s_{q'}^*u_h^{\phantom{*}}$ with $g,h \in G, p',q' \in P$, there is a unique $[(p,q)] \in P(\perp)/_\sim$ such that 
\[u_g^{\phantom{*}}s_{p'}^{\phantom{*}}s_{q'}^*u_h^{\phantom{*}} = s_p^*s_p^{\phantom{*}}u_g^{\phantom{*}}s_{p'}^{\phantom{*}}s_{q'}^*u_h^{\phantom{*}} s_q^*s_q^{\phantom{*}} = s_p^*F_{(p,q)}(u_g^{\phantom{*}}s_{p'}^{\phantom{*}}s_{q'}^*u_h^{\phantom{*}})s_q^{\phantom{*}},\]
obtained as $p'P \cap q'P = p'pP, p'p=q'q$ (a least common multiple of $p'$ and $q'$). Thus for every $a \in {\rm span}\{ u_g^{\phantom{*}}s_{p'}^{\phantom{*}}s_{q'}^*u_h^{\phantom{*}} \mid g,h \in G, p',q' \in P\}$, there is a uniquely determined finite set $A(a) \subset  P(\perp)/_\sim$ with the property that 
\begin{equation}\label{eq:F_(p,q) reconstruction formula}
\begin{array}{c} a = \sum\limits_{[(p,q)] \in A(a)} s_p^*F_{(p,q)}(a)s_q^{\phantom{*}}, \end{array}
\end{equation} 
which is well defined by \eqref{eq:F_(p,q) and sim}.
\end{remark}

\begin{lemma} \label{lem:Cuntzlemma}
If $a\in \CQ(\gxp)$ satisfies $F_{(p,q)}(a)=0$ for all $(p,q) \in P(\perp)$, then $a=0$.
\end{lemma}
\begin{proof}
Suppose $a\in \CQ(\gxp)$ satisfies $F_{(p,q)}(a)=0$ for all $(p,q) \in P(\perp)$. We fix a faithful representation $\varphi\colon \CQ(\gxp) \to \CL(H)$ on some Hilbert space $H$. Precomposing by a gauge automorphism $\varphi \circ \gamma_\chi \colon \CQ(\gxp) \to \CL(H)$ is then still a faithful representation for every $\chi \in T$. Given $\xi,\eta \in H$ with $\|\xi\|=\|\eta\|=1$, define $f\colon T \to \C$ by 
\[f(\chi) \doteq \langle \varphi(\gamma_\chi(a))\xi,\eta\rangle.\]
Pick a sequence $(a_k)_{k \in \N} \subset {\rm span}\{ u_g^{\phantom{*}}s_{p'}^{\phantom{*}}s_{q'}^*u_h^{\phantom{*}} \mid g,h \in G, p',q' \in P\}$ with $a_k \to a$, and define $f_k\colon T \to \C$ by  
\[f_k(\chi) \doteq \langle \varphi(\gamma_\chi(a_k))\xi,\eta\rangle.\]
As $\|f-f_k\|_\infty \leq \|a-a_k\| \to 0$,
$(f_k)_{k \in \N}$ converges uniformly to $f$ on $T$. According to \eqref{eq:F_(p,q) reconstruction formula}, there is a uniquely determined sequence of finite subsets $(A(a_k))_{k \in \N} \subset P(\perp)/_\sim$ such that $a_k = \sum_{[(p,q)] \in A(a_k)} s_p^*F_{(p,q)}(a_k)s_q^{\phantom{*}}$ for all $k \geq 1$. This leads us to 
\[\begin{array}{c}
f_k(\chi) = \sum\limits_{[(p,q)] \in A(a_k)} \langle \varphi(\gamma_\chi(s_p^*F_{(p,q)}(a_k)s_q^{\phantom{*}}))\xi,\eta\rangle = \sum\limits_{[(p,q)] \in A(a_k)} c_{p^{-1}q,k} \chi(p^{-1}q)
\end{array}\] 
for $c_{p^{-1}q,k} \doteq \langle \varphi(s_p^*F_{(p,q)}(a_k)s_q^{\phantom{*}})\xi,\eta\rangle$ as $F_{(p,q)}(a_k) \in \CF$. We set $c_{p^{-1}q,k}\doteq 0$ for $[(p,q)] \notin A(a_k)$. Note that if $[(p_1,q_1)],[(p_2,q_2)] \in A(a_k)$ are distinct, then $p_1^{-1}q_1^{\phantom{1}} \neq p_2^{-1}q_2^{\phantom{1}}$ in the group $P^{-1}P$ as $(p_i,q_i) \in P(\perp)$ for $i=1,2$. Therefore, we can interpret $c_{p^{-1}q,k}$ as the Fourier coefficient of $a_k$ for $p^{-1}q$. Every element $g \in P^{-1}P$ can be described as $p^{-1}q$ for some $(p,q) \in P(\perp)$ by removing the GCD from any given expression as a quotient of two elements from $P$. Since $(f_k)_{k \in \N}$ converges uniformly to $f$, the Fourier coefficients $(c_{p^{-1}q,k})_{k \in \N}$ converge to the Fourier coefficients $c_{p^{-1}q} \doteq \langle \varphi(s_p^*F_{(p,q)}(a)s_q^{\phantom{*}})\xi,\eta\rangle$ of $f$ for all $p^{-1}q \in P^{-1}P$. But then $F_{(p,q)}(a)=0$ for all $(p,q) \in P(\perp)$ forces $c_{p^{-1}q,k} \to c_{p^{-1}q} = 0$ for all $p^{-1}q \in P^{-1}P$, so that $f(\chi) = \lim_{k \to \infty} f_k(\chi) = 0$ for all $\chi \in T$. As $\xi$ and $\eta$ were arbitrary and $\varphi\circ\gamma_\chi$ was faithful, we get $a=0$.  
\end{proof}

We denote by $C^*(P^*)$ the $C^*$-algebra of the abelian group $P^*$, and think of $C^*(P^*)$ as the subalgebra of $\CQ(\gxp)$ generated by the unitaries $s_p, p \in P^*$.

\begin{lemma}\label{lem:unitaries in rel commutant}
Suppose that $P^*$ is finite. If $w\in \CQ(\gxp)$ is a unitary satisfying $w s_p w^*=z_p s_p$ with $z_p\in\IT$ for all $p\in P$, then $w\in C^*(P^*)$ and $z_p=1$ for all $p \in P$.
\end{lemma}
\begin{proof}
For every $r \in P$, \eqref{eq:(p,q) in P(perp) -> doubly commute} allows us to compute
\[s_r^*\gamma_\chi(s_p^{\phantom{*}}ws_q^*)s_r^{\phantom{*}} = \overline{z}_p\overline{z}_q\gamma_\chi(s_r^*s_q^*ws_p^{\phantom{*}}s_r^{\phantom{*}}) = z_r\overline{z}_p\overline{z}_q\gamma_\chi(s_q^*s_r^*s_r^{\phantom{*}}ws_p^{\phantom{*}}) = z_r\gamma_\chi(s_p^{\phantom{*}}ws_q^*)\]
for all $(p,q) \in P(\perp)$. Therefore we get $s_r^*F_{(p,q)}(w)s_r^{\phantom{*}} = z_r F_{(p,q)}(w)$ for all $r \in P$ and all $(p,q) \in P(\perp)$. As $F_{(p,q)}(w) \in \CF$ by definition, Lemma~\ref{lem:limfixpoint} implies that 
\[(s_{p_m}^*F_{(p,q)}(w)s_{p_m}^{\phantom{*}})_{m \in \N} = (z_{p_m})_{m \in \N} F_{(p,q)}(w)\] 
converges to a limit in $\C$. However, for every $(p,q) \in P(\perp)$ with $p \notin P^*$ or $q \notin P^*$, \eqref{eq:F_(p,q) range projections} forces $F_{(p,q)}(w) = 0$ as $s_p^{\phantom{*}}s_p^*$ or $s_q^{\phantom{*}}s_q^*$ is a proper subprojection of $1$ and hence not in $\C$. Since $P^*$ is finite, we can consider the element $a\doteq w - \sum_{\substack{ [(p,q)] \in P(\perp)/_\sim, \\ p,q \in P^*}} s_p^*F_{(p,q)}(w)s_q^{\phantom{*}}$, which then satisfies $F_{(p,q)}(a) = 0$ for all $(p,q) \in P(\perp)$. Lemma~\ref{lem:Cuntzlemma} now implies
\[\begin{array}{c} w = \sum\limits_{\substack{ [(p,q)] \in P(\perp)/_\sim, \\ p,q \in P^*}} s_p^*F_{(p,q)}(w)s_q^{\phantom{*}} \in C^*(P^*)\end{array}\] 
as $F_{(p,q)}(w) \in \C$ by the previous part. Since $P$ is abelian, $w$ commutes with every $s_p$, so $z_p=1$ for all $p \in P$.
\end{proof}

We are now ready to prove the announced theorem.

\begin{thm}\label{thm:rel commutant}
Suppose that $P^*$ is finite. The relative commutant $C^*(\{s_p \mid p \in P\})' \cap \CQ(\gxp)$ equals $C^*(P^*)$. In particular, if there are no non-trivial elements in $P^*$, then $C^*(\{s_p \mid p \in P\})' \cap \CQ(\gxp)=\IC$.
\end{thm}
\begin{proof}
Being a unital $C^*$-algebra, the relative commutant $C^*(\{s_p \mid p \in P\})' \cap \CQ(\gxp)$ is the linear span of its unitaries. According to Lemma~\ref{lem:unitaries in rel commutant}, every such unitary belongs to $C^*(P^*)$. The reverse inclusion is clear as $P$ is abelian.
\end{proof}

\begin{remark}\label{rem:rel commutant finite P*}

The assumption in Theorem~\ref{thm:rel commutant} that the subgroup $P^*$ of the invertible elements in $P$ needs to be finite
is likely to be dispensable. For instance, this hypothesis could be got rid of by establishing a Fej\'{e}r-type theorem 
for the series  
\[\begin{array}{c} 
\sum\limits_{\substack{ [(p,q)] \in P(\perp)/_\sim, \\ p,q \in P^*}} s_p^*F_{(p,q)}(w)s_q^{\phantom{*}} \end{array}\] 
which will in general fail to be norm convergent, even if $w$ is as above. Nevertheless, we would expect the series to be Ces\`{a}ro summable once a suitable way to count the elements of $P^*$ is introduced. This aspect lies outside the scope of the present work, but we plan to address it 
in future studies.
\end{remark}

\begin{remark}
In the context of integral dynamics, see Example~\ref{ex:integral dynamics}, the relative commutant $C^*(\{s_p \mid p \in T\})'\cap \CQ(\IZ\rtimes P)$ for $T \subset S$ might also be worth computing. We would expect it to equal  $C^*(\{s_p \mid p \in S\setminus T\})$.
\end{remark}

\section{Extendability of Bogolubov automorphisms}\label{sec:Bogolubov}
For every $2 \leq p < \infty$ the $C^*$-algebra  $\CQ_p= \CQ(\IZ \rtimes_{\theta_p} \IN)$ is by now known as the $p$-adic ring $C^*$-algebra. Inside each $\CQ_p$, there is a copy of the Cuntz algebra $\CO_p$, generated by $(u^ i s_p)_{0 \leq i \leq p-1}$. For convenience, let us from now on denote these generating isometries $u^is_p$ of $\CO_p\subset \CQ_p$ by $T_i$. The case $p=2$ has already been studied in detail in \cites{LarsenLi,ACR}. A special feature of $\CQ_2$ is that the two isometries $T_0$ and $T_1$ are intertwined by $u$, i.e. $T_0u=uT_1$. For a general $p$, however, $T_0$ and $T_{p-1}$ are still intertwined by $u$, but we cannot expect all the isometries $T_i$ to be unitarily equivalent to one another. For instance, consider $p=3$ and the faithful canonical representation $\pi\colon \CQ_3 \to \CL(\ell^2(\Z))$. Here, both $\pi(T_0)$ and $\pi(T_2)$ have eigenvalue $1$ corresponding to the eigenvectors $\xi_0$ and $\xi_{-1}$, respectively. However, the point spectrum of $\pi(T_1)$ is empty. Indeed, if $\xi = \sum_{m\in\Z} c_m \xi_m \in \ell^2(\IZ)$ satisfies $\pi(T_1)\xi = \lambda\xi$ for some $\lambda \in \IT$, then $\sum_{m\in\Z} c_m \xi_{3m+1}=\sum_{m\in\Z} \lambda c_m \xi_m$. Therefore, $|c_m|=|c_{3m+1}|$ for all if $m \in \Z$, which forces $\xi=0$.

We start our discussion by introducing a distinguished representation of $\CO_p$, occasionally referred to as the \emph{interval picture of $\CO_p$}, which will come in useful in Lemma~\ref{lem:extensibility of Bogolubov automorphisms}. For $2 \leq p < \infty$, we define a map $\sigma\colon\CO_p \to \CL(L^2([0,1]))$ by $\sigma(T_i)(f)(t) = \sqrt{p} f \circ h_i^{-1}(t)$ for $t \in [i/p,(i+1)/p] $ and 0 otherwise, for $f\in L^2([0,1])$ and $i=0, \ldots, p-1$, where $h_i$ is the compression
\[\begin{array}{rcl}
h_i\colon [0,1] &\to& [i/p,(i+1)/p] \\
t &\mapsto& (i+t)/p
\end{array}\]
and the Hilbert space $L^2([0,1])$ is defined w.r.t. the Lebesgue measure. Observing that the adjoint of $\sigma(T_i)$ is given by $\sigma(T_i)^*(f) = \frac{1}{\sqrt{p}} f\circ h_i $, we see that $\sigma(T_i)^*\sigma(T_i) = \delta_{i,j} \id_{L^2([0,1])}$ and $\sigma(T_i)\sigma(T_i)^* = \id_{L^2([i/p,(i+1)/p])}$. Hence $\sigma$ defines a representation of $\CO_p$.

\begin{remark}\label{rem:T_p-1 pure means at most one extension}
If $\varphi\colon \CO_p \to \CL(\CH)$ is a unital representation on some Hilbert space $\CH$ such that $\varphi(T_{p-1})$ is a pure isometry (namely the range projections of its powers converge strongly to $0$), then $\varphi$ has at most one extension to a representation of $\CQ_p$. Obviously, any extension is completely determined by the image of the unitary $u$. 
Now if $U,W \in \CU(\CH)$ are such that either of them yields an extension of $\varphi$ to a representation of $\CQ_p$, then $W\varphi(T_i) = \varphi(T_{i+1}) = U\varphi(T_i)$ for all $0\leq i \leq p-2$, $W\varphi(T_{p-1}) = \varphi(T_0)W$, and $U\varphi(T_{p-1}) = \varphi(T_0)U$ lead to 
\[\begin{array}{c} U-W = \lim\limits_{k \to \infty} \sum\limits_{|\alpha|=k}(U-W)\varphi(T_{\alpha})\varphi(T_{\alpha})^* = \lim\limits_{k \to \infty} (U-W)\varphi(T_{p-1})^{k}\varphi(T_{p-1})^{*k} =0, \end{array}\]
where $\alpha$ ranges over multi-indices of length $k$ in the letters $\{0,1,\ldots,p-1\}$ and $T_\alpha = T_{\alpha_1}T_{\alpha_2}\cdots T_{\alpha_k}$, because $\varphi(T_{p-1})$ is pure. 
\end{remark}

\begin{proposition}\label{prop:interval rep from O_p to Q_p}
For $2 \leq p < \infty$, the isometries $(\sigma(T_i))_{0 \leq i \leq p-1}$ are pure and the representation $\sigma$ of $\CO_p$ extends uniquely to a representation $\widetilde{\sigma} \colon \CQ_p \to \CL(L^2([0,1]))$. Moreover, $\widetilde{\sigma}$ is not unitarily equivalent to the canonical representation $\pi\colon \CQ_p \to \CL(\ell^2(\Z))$. 
\end{proposition}
\begin{proof}
We start by observing that the isometry $\sigma(T_{i})$ is pure since $\sigma(T_i)^k\sigma(T_i)^{*k}$ is the projection onto $L^2(I)$, where $I \subset [0,1]$ is an interval of length $p^{-k}$. According to Remark~\ref{rem:T_p-1 pure means at most one extension} $\sigma$ thus admits at most one extension to a representation of $\CQ_p$. In order to obtain an extension $\widetilde{\sigma}$, we note that the each of the following two families of mutually orthogonal projections  
\[\{\sigma(T_{p-1}^j T_i^{\phantom{*}} T_i^* T_{p-1}^{*j}) \mid 0\leq i \leq p-2, j\geq 0\} \quad\text{and}\quad \{\sigma(T_0^j T_i^{\phantom{*}} T_i^* T_0^{*j}) \mid 1 \leq i \leq p-1, j\geq 0\}\] 
sums up to the identity due to the pureness of $T_{p-1}$ and $T_0$, respectively. We then define $U \in \CL(L^2([0,1]))$ by 
\begin{equation}\label{eq:u for sigma}
U\sigma(T_{p-1}^j T_i^{\phantom{*}} T_i^* T_{p-1}^{*j}) \doteq \sigma(T_{0}^j T_{i+1}^{\phantom{*}}T_i^* T_{p-1}^{*j}) \quad \text{for } 0\leq i \leq p-2, j\geq 0.
\end{equation}
As 
\[\begin{array}{c}
UU^* = \sum\limits_{\substack{1\leq i \leq p-1,\\ j \geq 0}} \sigma(T_0^j T_i^{\phantom{*}} T_i^* T_0^{*j}) = 1, \quad\text{and} \vspace*{2mm}\\
U^*U = U^*\sum\limits_{\substack{1\leq i \leq p-1,\\ j \geq 0}} \sigma(T_0^j T_i^{\phantom{*}} T_i^* T_0^{*j})U = \sum\limits_{\substack{0\leq i \leq p-2,\\ j \geq 0}} \sigma(T_{p-1}^j T_i^{\phantom{*}} T_i^* T_{p-1}^{*j}) = 1,
\end{array}\] 
the operator $U$ is a unitary. For $i=0,\ldots,p-2$, we clearly have $U\sigma(T_i) = \sigma(T_{i+1})$ by \eqref{eq:u for sigma}, and $U\sigma(T_{p-1}) = \sigma(T_0)U$ follows from
\[\begin{array}{c}
U\sigma(T_{p-1})U^* = U \sum\limits_{\substack{0\leq i \leq p-2,\\ j \geq 0}} \sigma(T_{p-1}^{j+1} T_i^{\phantom{*}} T_i^* T_{p-1}^{*j})U^* = \sum\limits_{\substack{1\leq i \leq p-1,\\ j \geq 0}} \sigma(T_0^{j+1} T_i^{\phantom{*}} T_i^* T_0^{*j}) = \sigma(T_0).
\end{array}\]
Finally, the representation $\widetilde{\sigma}$ is not unitarily equivalent to $\pi$ because the point spectrum of the pure isometries is empty, whereas $\pi(T_{p-1})$ has eigenvalue $1$ corresponding to the eigenspace spanned by $\xi_{-1} \in \ell^2(\Z)$.
\end{proof}

\begin{remark}
We provide an explicit description of the operator $U$ for $p=2$. As it turns out to be defined piecewise, we need to set some notation.
We define $I_0\doteq[0,\frac{1}{2}]$, $I_k\doteq[\sum_{l=1}^k\frac{1}{2^l},\sum_{l=1}^{k+1}\frac{1}{2^l}]$, $k\geq 1$, and
$J_k\doteq[\frac{1}{2^{j+1}}, \frac{1}{2^{j}}]$, $k\geq 0$.
Given any $f\in L^2([0,1])$ we can write $f=\sum_{k=0}^\infty f_k$, where $f_k= f\chi_{I_k}$. Then $Uf$ is given by the function $g=\sum_{k=0}^\infty g_k$, where $g_k = g\chi_{J_k}$ with
$$
\begin{cases} 
g_0(s)=f_0(s-\frac{1}{2}) & s	\in [\frac{1}{2},1] \\
g_1(s)=f_1(s+\frac{1}{4}) & s	\in [\frac{1}{4},\frac{1}{2}] \\
g_k(s)=f_k(s+\sum_{l=1}^{k-1}\frac{1}{2^l}+ \frac{1}{2^{k+1}}) & s	\in J_k, k\geq 2 \\
 \end{cases}
$$
\end{remark}

In the sequel, we shall refer to the representation $\widetilde{\sigma}\colon \CQ_p \to \CL(L^2([0,1]))$ from Proposition~\ref{prop:interval rep from O_p to Q_p} as the \emph{interval picture of $\CQ_p$}.\vspace*{3mm}\\

A distinguished family of automorphisms of $\CO_p$ is given by the Bogolubov automorphisms $\lambda_A$ associated with a unitary matrix $A=(a_{i,j})\in \CU_p(\IC)$, where $\lambda_A(T_i) \doteq \sum_{j=0}^{p-1} a_{j,i} T_j$. As $\lambda_A\circ\lambda_B=\lambda_{AB}$, the map $\CU_p(C)\ni A\mapsto\lambda_A\in\Aut(\CO_p)$ is a representation of the group $\CU_p(\IC)$. Gauge automorphisms are obviously a special case of Bogolubov automorphisms, i.e. those coming from diagonal matrices with entries in $\IT$. The aim of this section is to determine which Bogolubov automorphisms of $\CO_p$ extend to automorphisms of $\CQ_p$. In addition, we show that all these Bogolubov automorphisms admit a unique extension to an automorphism of $\CQ_p$, see Theorem~\ref{thm:extensible Bogolubov autos}: The group of extendible Bogolubov automorphisms is generated by the gauge automorphisms and the exchange automorphism described in  Definition~\ref{def:exchange automorphism}. Moreover, we show that the extensions of all nontrivial gauge automorphisms of $\CO_p$ are outer automorphisms of $\CQ_p$, see Theorem~\ref{thm:gaugeQpouter}.

\begin{lemma}\label{lem:extensibility of Bogolubov automorphisms}
For all $2\leq p < \infty$, every Bogolubov automorphism $\lambda_A$ of $\CO_p$ admits at most one extension $\widetilde{\lambda_A}$ to an automorphism of $\CQ_p$.
\end{lemma}
\begin{proof}
Let $A = (a_{i,j}) \in \CU_p(\C)$. We intend to invoke Remark~\ref{rem:T_p-1 pure means at most one extension}, so we need to work in a representation $\widetilde{\varphi}$ of $\CQ_p$ that extends a non-zero representation $\varphi$ of $\CO_p$ for which $\varphi(\lambda_A(T_{p-1}))$ is pure. For if $\widetilde{\lambda_A}$ is an extension of $\lambda_A$ to $\CQ_p$, then $\widetilde{\varphi}\circ\widetilde{\lambda_A}$ will be an extension of $\varphi\circ\lambda_A$, which is unique due to Remark~\ref{rem:T_p-1 pure means at most one extension}. But if $\varphi$ is non-zero, it is actually an isomorphism, and so is $\widetilde{\varphi}$ as both $\CO_p$ and $\CQ_p$ are simple. Therefore, the extension $\widetilde{\lambda_A}$ is also unique.

If $M\doteq\max\{|a_{i, p-1}| \mid i=0,1,\ldots p-1\} \in [0,1]$ equals $1$, then $a_{i,p-1} = \delta_{i,i_0}$ for a unique  $0 \leq i_0 \leq p-1$. According to Proposition~\ref{prop:interval rep from O_p to Q_p}, we can then pick $\varphi = \sigma$. The case $M<1$ is handled in the canonical representation $\varphi = \pi$ instead: Let $Q_k\doteq \pi(\lambda_A(T_{p-1}^kT_{p-1}^{*k}))$ denote the range projection of the $k$-th power. We observe that $\lambda_A(T_{p-1})^k=\sum_{|\alpha|=k} c_\alpha T_\alpha$, where $\alpha$ is a multi-index with values in $\{0,1,\ldots,p-1\}$ and $c_\alpha =a_{j_1, p-1}\ldots a_{j_k, p-1}$ for $\alpha = (j_1,\ldots, j_k)$. Then the inequality $|c_\alpha |\leq M^k$ forces $\|Q_k \xi_m\| \leq M^k \to 0$ as $k\to \infty$ for every $m \in \Z$ because $M<1$ and $T_\alpha^*T_\beta^{\phantom{*}} = 0$ for all multi-indices $\alpha \neq \beta$ of length $k$. Therefore, $\pi(\lambda_A(T_{p-1}))$ is pure and the proof is complete.
\end{proof}

We will now show that very few Bogolubov automorphisms are extendible. 

\begin{lemma}\label{lem:Bogolubov extensions preserve C*(Z)}
If $\lambda_A$ is an extendible Bogolubov automorphism, then $\widetilde{\lambda_A}(u) \in C^*(\Z)$.
\end{lemma}
\begin{proof}
As an intermediate step, we show that $\widetilde{\lambda_A}(u) \in \CF$, for which we proceed by contradiction. Suppose there exists some $z\in \IT\setminus\{1\}$ such that $\gamma_z(\widetilde{\lambda_A}(u)) \neq \widetilde{\lambda_A}(u)$. Set $\Lambda (s_p)\doteq \lambda_A(s_p)$ and $\Lambda (u)\doteq \gamma_z(\widetilde{\lambda_A}(u))$. We want to show that $\Lambda$ is an automorphism of $\CQ_p$ that extends $\lambda_A$. The calculation
\[\Lambda(T_i) = \Lambda(u^i)\Lambda(s_p) = \gamma_z(\widetilde{\lambda_A}(u^i))\lambda_A(s_p) = \overline{z}\gamma_z(\lambda_A(T_i)) = \lambda_A(T_i) \]
for $T_i = u^is_p$ ensures that $\Lambda|_{\CO_p} = \lambda_A$. The defining relation of $\CQ_p$ corresponding to (I) and (III) are satisfied as
\[\begin{array}{l}
\Lambda(s_p) \Lambda(u) = \overline{z}\gamma_z(\widetilde{\lambda_A}(s_pu)) = \overline{z}\gamma_z(\widetilde{\lambda_A}(u^ps_p)) = \Lambda(u)^p\Lambda(s_p), \text{ and} \vspace*{2mm}\\
\sum\limits_{m=0}^{p-1}\Lambda(u)^m\Lambda(s_p)\Lambda(s_p)^* \Lambda(u)^{-m} 
= \sum\limits_{m=0}^{p-1}z\overline{z} \gamma_z(\lambda_A(e_{m,p})) = \gamma_z(\lambda_A(1)) = 1. 
\end{array}\]
Now Lemma~\ref{lem:extensibility of Bogolubov automorphisms} yields a contradiction. Hence we get $\widetilde{\lambda_A}(u) \in \CF$.

Next, we note that for all $x\in \CF'_k \doteq {\rm span}\{ u^ms_{p^k}^{\phantom{*}}s_{p^k}^*u^{-n} \mid m,n \in \Z\}$, the element $\lambda_A(T_0^{*k}) x\lambda_A(T_{p-1}^k)$ belongs to $C^*(\IZ)$. Let $x \in \CF = \overline{\bigcup_{k \geq 1}\CF_k'}$ and choose an approximating sequence $(x_k)_{k \geq 1}$ with $x_k \in \CF'_{f(k)}$ for some $f\colon \N\to\N$. As $\CF_k' \subset \CF_{k+1}'$ for all $k$, we can assume that $f$ is monotone increasing. Now if $\lambda_A(T_0^{*k}) x\lambda_A(T_{p-1}^k) \to y$ in $\CQ_p$, then $y \in C^*(\Z)$ as
\[\begin{array}{lcl}
\| y-\lambda_A(T_0^{*f(k)}) x_k \lambda_A(T_{p-1}^{f(k)})\|
&\leq& \| y-\lambda_A(T_0^{*f(k)}) x \lambda_A(T_{p-1}^{f(k)})\| + \| x-x_k\| \to 0
\end{array}\]
and $\lambda_A(T_0^{*f(k)}) x_k \lambda_A(T_{p-1}^{f(k)}) \in C^*(\Z)$. Now, we observe that $T_0^{*k}uT_{p-1}^k = u$. Thus, using that $\widetilde{\lambda_A}(u) \in \CF$ by the first part and that $(\lambda_A(T_0^{*k}) \widetilde{\lambda_A}(u)\lambda_A(T_{p-1}^k))_{k \geq 1}$ is a constant sequence equal to $\widetilde{\lambda_A}(u)$, we deduce that $\widetilde{\lambda_A}(u) \in C^*(\Z)$.
\end{proof}

\begin{definition}\label{def:exchange automorphism}
For the \emph{exchange matrix} $E \in M_p(\{0,1\})$, that is, $E= (\delta_{p-i+1,j})_{1 \leq i,j \leq p}$, the associated Bogolubov automorphism $\lambda_E$ of $\CO_p$ is called the \emph{exchange automorphism}.
\end{definition}

For the proof of our main result Theorem~\ref{thm:extensible Bogolubov autos}, we need the following small variation of \cite{ACR}*{Proposition~A.1}, for which we remark that the proof carries over verbatim. Here we limit ourselves to pointing out that the integer $n$ is the winding number of $f$, which is well defined thanks to compactness of $\IT$ along with continuity of $f$.

\begin{proposition}\label{prop:functional equation monomial solutions}
For all $2\leq p <\infty$, every $f \in C(\IT,\IT)$ satisfying $f(z^p)=f(z)^p$ for all $z\in\IT$ is of the form $f(z)=z^n$ for some $n \in \Z$.
\end{proposition}

\begin{thm}\label{thm:extensible Bogolubov autos}
A Bogolubov automorphism $\lambda_A$ of $\CO_p$ admits an extension to an automorphism of $\CQ_p$ if and only if $A$ belongs to the subgroup $\{ zB \mid z \in \IT, B \in \{1,E\}\}\cong \IT\times \Z/2\Z$ of $\CU_p(\C)$. If $A=z1$, then $\widetilde{\lambda_A}(u) = u$, whereas $A=zE$ implies $\widetilde{\lambda_A}(u) = u^*$.
\end{thm}
\begin{proof}
By Lemma~\ref{lem:Bogolubov extensions preserve C*(Z)}, there is $f \in C(\IT,\IT) \cong \CU(C^*(\Z))$ such that $f(u) = \widetilde{\lambda_A}(u)$. Applying $\widetilde{\lambda_A}$ to both sides of (I) and using the fact that $f(u^p) s_p = s_p f(u)$, we easily get
\[\begin{array}{lclclcl}
f(u)^p \lambda_A(T_i) &=& \widetilde{\lambda_A}(u)^{p+i}\lambda_A(T_0) &=& \widetilde{\lambda_A}(u)^i \lambda_A(T_0) \widetilde{\lambda_A}(u) &=& \widetilde{\lambda_A}(u)^i \lambda_A(T_0) f(u) \\
&&&=& \widetilde{\lambda_A}(u)^i f(u^p)\lambda_A(T_0) &=& f(u^p) \lambda_A(T_i)
\end{array}\]
for all $0 \leq i \leq p-1$. Thus relation (III) implies that $f$ satisfies the functional equation from Proposition~\ref{prop:functional equation monomial solutions} for $p$, and hence there is $n \in \Z$ such that $f(z)=z^n$ for all $z \in \IT$. 

Next, we observe that each $1 \leq i \leq p-1$ yields an equation
\[\begin{array}{c}
\sum\limits_{j=0}^{p-1}a_{j,i} u^j s_p= \sum\limits_{j=0}^{p-1}a_{j,i} T_j = \lambda_A(T_i)=\widetilde{\lambda_A}(u^i)\lambda_A(T_0) = \sum\limits_{j=0}^{p-1}a_{j,0} u^{j+ni} s_p,
\end{array}\]
which, applied to $\xi_0$ gives
\begin{equation}\label{eq:bogo}
\begin{array}{c} \sum\limits_{j=0}^{p-1}a_{j,i}\xi_j = \sum\limits_{k=0}^{p-1}a_{k,0} \xi_{k+ni} \quad \text{for all } 1 \leq i \leq p-1.\end{array}
\end{equation}

\begin{description}
\item[$n=0$] The algebra $\CQ_p$ is simple and $\widetilde{\lambda_A} \neq 0$, so $\widetilde{\lambda_A}$ has to be faithful, which excludes $n=0$.
\item[$|n| \geq 2$] The fact that $j \neq k+n(p-1)$ for all $0\leq j,k \leq p-1$ forces $\lambda_A(T_{p-1}) = 0$ due to \eqref{eq:bogo}, but this is impossible as $\lambda_A(T_{p-1})$ needs to be an isometry.  
\item[$n=1$] As $j \neq k+n(p-1)$ for all $0\leq j,k \leq p-1$ except $j=p-1$ and $k=0$, \eqref{eq:bogo} implies that $a_{k,0} = z \delta_{k,0}$ for some $z \in \IT$, as $\lambda_A(T_0)$ needs to be an isometry. But then we get 
\[\lambda_A(T_i) = \widetilde{\lambda_A}(u^i) \lambda_A(T_0) \stackrel{n=1}{=} zu^i T_0 = \gamma_z(T_i)\]
for all $1 \leq i \leq p-1$, so that $\lambda_A = \gamma_z$.
\item[$n=-1$] Similar to the case of $n=1$, we get $a_{k,0} = z \delta_{k,p-1}$ for some $z \in \IT$ by looking at $i=p-1$. From this we get
 \[\lambda_A(T_i) = \widetilde{\lambda_A}(u^i) \lambda_A(T_0) \stackrel{n=-1}{=} zu^{-i} T_{p-1} = \gamma_z(\lambda_E(T_i))\]
for all $1 \leq i \leq p-1$, so that $\lambda_A = \gamma_z \circ \lambda_E$.
\end{description}   
\end{proof}

\begin{remark}\label{rem:isometries fixed by extensible Bogolubov}
According to Theorem~\ref{thm:extensible Bogolubov autos}, the only way an extendible non-trivial Bogolubov automorphism $\lambda_A$ can fix one of the generating isometries $(T_i)_{0 \leq i \leq p-1}$ is that $p$ is odd and $A$ is the exchange matrix $E$, in which case the fixed isometry is $T_{(p-1)/2}$.
\end{remark}

\begin{remark}\label{rem:Bogolubov autom beyond Q_p}
In the case of algebraic dynamical systems $(\gpt)$ where $P$ has more than one generator, it is not clear what the right definition of a Bogolubov automorphism ought to be. A good place to start appears to be the case of integral dynamics, see Example~\ref{ex:integral dynamics}, the natural higher-dimensional version of $\CQ_p$ where $P \subset \N^\times$ is generated by a family $S$ of mutually relatively prime natural numbers. In such a situation, there is a notion of a torsion subalgebra $\CA_S = C^*(\{u^is_p \mid p \in S, 0\leq i\leq p-1\}) \subset \CQ(\Z\rtimes P)$, which plays the role of $\CO_p \subset \CQ_p$ in many ways, see \cite{BOS1} for details.
\end{remark}

\section{Automorphisms preserving the group C*-algebra}\label{sec:preserving C*(G)}
\subsection{Automorphisms fixing the group C*-algebra}\label{subsec:aut fix C*(G)}
In this subsection we consider automorphisms of $\CQ(\gxp)$ that fix the group $C^*$-algebra $C^*(G) \subset \CQ(\gxp)$ pointwise. The corresponding subgroup of ${\rm Aut} \CQ(\gxp)$ shall be denoted by ${\rm Aut}_{C^*(G)}\CQ(\gxp)$. We restrict our attention to the situation where $(\gpt)$ is an algebraic dynamical system of finite type such that $G$ and $P$ are abelian. We note that this covers both Example~\ref{ex:abelian group - single finite endo} and Example~\ref{ex:integral dynamics}. In Theorem~\ref{thm:aut of Q fixing C*(G)} we show that, under this hypothesis, every element of ${\rm Aut}_{C^*(G)}\CQ(\gxp)$ arises from a suitable family of unitaries in the commutative group $C^*$-algebra $C^*(G)$, given that $C^*(G)$ is maximal abelian in $\CQ(\gxp)$. This result generalizes \cite{ACR}*{Theorems~6.13 and 6.14}, and is used in Theorem~\ref{thm:quasi-free-max-1} to show that ${\rm Aut}_{C^*(G)}\CQ(\gxp)$ is a maximal abelian subgroup of the automorphism group of $\CQ(\gxp)$. These results point towards a generalization of the notion of quasi-freeness that appears in \cites{Dykema,Zach}, but we shall not pursue this line of research here. 

In the sequel, we will make use of commutativity of $G$ to identify $C^*(G)$ with $C(\widehat{G})$ via $u_g \mapsto [\chi \mapsto \chi(g)]$, where $\widehat{G}$ is the Pontryagin dual of $G$.

\begin{definition}\label{def:twisted Hom(P,U(C*(G)))}
A map $\psi\colon P \to C(\widehat{G},\IT), p \mapsto \psi_p$ is said to be a \emph{$\theta$-twisted homomorphism} if it satisfies 
\begin{equation}\label{eq:funct equation abstract}
\psi_{pq} = \psi_p\tilde{\theta}_p(\psi_q) \quad \text{for all } p,q \in P,
\end{equation}
where $\tilde{\theta}_p(u_g) \doteq u_{\theta_p(g)}$ for $g \in G$. The collection of all $\theta$-twisted homomorphisms is denoted by ${\rm Hom}_\theta(P,C(\widehat{G},\IT))$.
\end{definition}

We note that ${\rm Hom}_\theta(P,C(\widehat{G},\IT))$ is an abelian group under pointwise multiplication with inverses given by $(\psi^{-1})_p \doteq (\psi_p)^{-1}$ because $\theta$ consists of group homomorphisms of $G$. 
 
\begin{thm}\label{thm:aut of Q fixing C*(G)}
The map $\psi \mapsto \beta_\psi$ with $\beta_\psi(u_gs_p) \doteq u_g \psi_p s_p$ for $(g,p) \in \gxp$ defines an injective group homomorphism ${\rm Hom}_\theta(P,C(\widehat{G},\IT)) \to {\rm Aut}_{C^*(G)}\CQ(\gxp)$. If $C^*(G)$ is maximal abelian in $\CQ(\gxp)$, then this homomorphism is an isomorphism and
\[{\rm End}_{C^*(G)}\CQ(\gxp) = {\rm Aut}_{C^*(G)}\CQ(\gxp) \cong {\rm Hom}_\theta(P,C(\widehat{G},\IT)).\]
\end{thm} 
\begin{proof}
We first verify the defining relations (I)--(III) from Section~\ref{sec:prelim} for $\beta_\psi(u_g)=u_g$ and $\beta_\psi(s_p)$: Relation (I) holds as 
\[\beta_\psi(s_p)u_g = \psi_p s_pu_g = \psi_p u_{\theta_p(g)}s_p = u_{\theta_p(g)}\beta_\psi(s_p).\]
For (II), let $p,q \in P, g \in G$. Since $P$ is abelian, we let $r \in P$ denote a greatest common divisor of $p$ and $q$, see Proposition~\ref{prop:GCD and LCM for abelian right LCM}, and pick $q',p' \in P$ such that $p=rq', q=rp'$, so that $pP\cap qP = rq'p'P$. As $q'$ and $p'$ are relatively prime, that is, $(q',p') \in P(\perp)$ in the notation of Remark~\ref{rem:Fourier-maps}, and $(\gpt)$ is of finite type, \cite{Sta1}*{Proposition~1.1} implies $\theta_{q'}(G)\theta_{p'}(G)=G$. By \eqref{eq:funct equation abstract}, we have 
\[\psi_p^*\psi_q^{\phantom{*}} = \tilde{\theta}_r(\psi_{q'}^*)\psi_r^*\psi_r^{\phantom{*}}\tilde{\theta}_r(\psi_{p'}^{\phantom{*}}) = \tilde{\theta}_r(\psi_{q'}^*\psi_{p'}^{\phantom{*}})\] 
so that, after some computations,
\[\begin{array}{lclcl}
\beta_\psi(s_p^*)u_g^{\phantom{*}}\beta_\psi(s_q^{\phantom{*}}) &=& s_p^*\tilde{\theta}_r(\psi_{q'}^*\psi_{p'}^{\phantom{*}}) u_g^{\phantom{*}} s_q^{\phantom{*}} \vspace*{2mm}\\
 &=& \begin{cases} u_{g_1}^{\phantom{*}}s_{q'}^* \psi_{q'}^*\psi_{p'}^{\phantom{*}} s_{p'}^{\phantom{*}}u_{g_2}^{\phantom{*}} &,\text{ if } g = \theta_p(g_1)\theta_q(g_2) \ (\Leftrightarrow g \in \theta_r(G))\\ 0 &,\text {otherwise.} \end{cases}
\end{array}\]
Now we observe that $s_{q'}^* \psi_{q'}^*\psi_{p'}^{\phantom{*}} s_{p'}^{\phantom{*}}$ equals $\psi_{p'}^{\phantom{*}}s_{p'}^{\phantom{*}}s_{q'}^*\psi_{q'}^*$ as
\[\psi_{p'}^{\phantom{*}}s_{p'}^{\phantom{*}}s_{q'}^*\psi_{q'}^* = \psi_{p'}^{\phantom{*}}s_{q'}^*s_{p'}^{\phantom{*}}\psi_{q'}^* = s_{q'}^* \tilde{\theta}_{q'}(\psi_{p'}^{\phantom{*}})\tilde{\theta}_{p'}(\psi_{q'}^*) s_{p'}^{\phantom{*}}\]
and
\[\psi_{q'}^*\psi_{p'}^{\phantom{*}} = \tilde{\theta}_{q'}(\psi_{p'}^{\phantom{*}})\tilde{\theta}_{p'}(\psi_{q'}^*) \Longleftrightarrow \psi_{p'}^{\phantom{*}}\tilde{\theta}_{p'}(\psi_{q'}^{\phantom{*}}) = \psi_{q'p'}^{\phantom{*}} = \psi_{q'}^{\phantom{*}}\tilde{\theta}_{q'}(\psi_{p'}^{\phantom{*}}).\]
This shows that relation (II) is satisfied.
Checking the summation relation (III) amounts to 
\[\begin{array}{c}
\sum\limits_{\overline{g} \in G/\theta_p(G)} \beta_\psi(e_{g,p}^{\phantom{*}}) =  \psi_p^{\phantom{*}} \left(\sum\limits_{\overline{g} \in G/\theta_p(G)}e_{g,p}^{\phantom{*}} \right)\psi_p^* = 1
\end{array}\]
because $\psi_p$ is a unitary. Thus $\beta_\psi$ defines an endomorphism of $\CQ(\gxp)$. It is evident that $\beta_\psi \circ \beta_\varphi = \beta_{\psi\varphi}$. In particular, $\beta_{\psi^{-1}} \circ \beta_\psi = \id = \beta_\psi \circ \beta_{\psi^{-1}}$ implies that, $\beta_\psi$ is an automorphism, and hence an element of ${\rm Aut}_{C^*(G)}\CQ(\gxp)$ because $\beta_\psi(u_g) = u_g$ as $C^*(G)$ is abelian. 
If $\beta_\psi=\beta_\varphi$ for $\psi,\varphi \in {\rm Hom}_\theta(P,C(\widehat{G},\IT))$, then $\psi_ps_p = \beta_\psi(s_p) = \beta_\varphi(s_p) = \varphi_ps_p$. Multiplying the equation by $u_g$ on the left, by $s_p^*u_g^*$ on the right, where $g \in G$ ranges over a transversal for $G/\theta_p(G)$ leads to 
\[\begin{array}{c}
\psi_p = \psi_p \left(\sum\limits_{\overline{g} \in G/\theta_p(G)} e_{g,p}\right) = \varphi_p \left(\sum\limits_{\overline{g} \in G/\theta_p(G)} e_{g,p}\right) = \varphi_p.
\end{array}\]
Thus the group homomorphism  ${\rm Hom}_\theta(P,C(\widehat{G},\IT)) \to {\rm Aut}_{C^*(G)}\CQ(\gxp), \psi \mapsto \beta_\psi$ is injective.

As for the second claim, suppose $C^*(G)$ is maximal abelian in $\CQ(\gxp)$, and let $\beta \in {\rm End}_{C^*(G)}\CQ(\gxp)$. Then both parts of the claim will follow if we show $\beta=\beta_\psi$ for some $\psi \in {\rm Hom}_\theta(P,C(\widehat{G},\IT))$. Note that $\beta$ is necessarily unital as $1 \in C^*(G)$. For all $g,h \in G$ and $p \in P$ we have
\[s_p^*u_g^*\beta(s_p^{\phantom{*}})u_h = s_p^*u_g^*u_{\theta_p(h)}\beta(s_p^{\phantom{*}}) = u_hs_p^*u_g^*\beta(s_p^{\phantom{*}})\]
because $G$ is abelian and $\beta$ fixes $C^*(G)$ pointwise. Thus $w_{g,p}\doteq s_p^*u_g^*\beta(s_p^{\phantom{*}})$ is an element of $C^*(G)$, as the latter is assumed maximal abelian. Then the finite linear combination $w_p \doteq \sum_{\overline{g} \in G/\theta_p(G)} u_g\tilde{\theta}_p(w_{g,p}) \in C^*(G)$ satisfies
\[\begin{array}{c}
w_ps_p = \sum\limits_{\overline{g} \in G/\theta_p(G)} u_g\tilde{\theta}_p(w_{g,p})s_p = \sum\limits_{\overline{g} \in G/\theta_p(G)} u_g^{\phantom{*}} s_p^{\phantom{*}} s_p^* u_g^* \beta(s_p^{\phantom{*}}) = \beta(s_p^{\phantom{*}}).
\end{array}\]
The element $w_p$ is a unitary in $C^*(G)$ since $w_pu_gs_p = u_gw_ps_p = u_g\beta(s_p) = \beta(u_gs_p)$ and the summation relation for $p$ allow us to compute
\[\begin{array}{c}
w_p^*w_p^{\phantom{*}} = w_p^{\phantom{*}}w_p^* = w_p^{\phantom{*}} \left( \sum\limits_{\overline{g} \in G/\theta_p(G)} e_{g,p}^{\phantom{*}} \right)w_p^* = \beta\left(\sum\limits_{\overline{g} \in G/\theta_p(G)} e_{g,p}\right) = \beta(1) = 1.
\end{array}\] 
Finally, we claim that $\psi(p)\doteq w_p$ defines a $\theta$-twisted homomorphism so that $\beta = \beta_\psi$. Indeed, given $p,q \in P$ we deduce $\psi(pq) = \psi(p)\tilde{\theta}_p(\psi(q))$ from the equation
\[\psi(pq)s_{pq} = \beta(s_{pq}) = \beta(s_p)\beta(s_q) = \psi(p)\tilde{\theta}_p(\psi(q)) s_{pq}\]
in the same way as we proved injectivity of the group homomorphism. Thus we arrive at $\beta=\beta_\psi$.
\end{proof}

The proof of Theorem~\ref{thm:aut of Q fixing C*(G)} shows that a family of unitaries in $C^*(G)$ defines a $*$-homomorphism of $\CQ(\gxp)$ if and only if it comes from a $\theta$-twisted homomorphism. 
As a strengthening of Theorem~\ref{thm:aut of Q fixing C*(G)}, we will now prove that ${\rm Aut}_{C^*(G)}\CQ(\gxp)$  is not only abelian but maximal abelian in ${\rm Aut}\CQ(\gxp)$ in a number of important cases, see Theorem~\ref{thm:quasi-free-max-1}.

\begin{lemma} \label{lem:metgroup}
Let $K$ be a metrizable compact abelian group. If $\chi\in K$ has the property that for every $f\in C(K, \IT)$ there exists $\lambda_f \in \IT$ with $f(\chi \xi)=\lambda_f f(\xi)$ for all $\xi\in K$, then $\chi= 1_K$.
\end{lemma}
\begin{proof}
In fact, this is a particular case of a more general result proved in Theorem 3.6 of \cite{ACR2} that the same is true for every continuous map $\Phi$ on a metrizable compact space $X$.
\end{proof}

\begin{remark}\label{rem:metgroup idea of proof}
The result recalled in the proof of Lemma~\ref{lem:metgroup} extends to general $C^*$-algebras: every endomorphism $\varphi$ of a unital $C^*$-algebra $\FA$ such that $\varphi(u)=\chi_u u$ for every $u\in\CU({\FA})$, where $\chi_u\in\IT$, must be the identity map of $\FA$. This is seen as follows. If $a\in\FA_{sa}$ and $\|a\|\leq 1$, then $a=\frac{1}{2}(f_+(a)+f_-(a))$ with $f_+(t)=t+i\sqrt{1-t^2}$ and $f_-(t)=t-i\sqrt{1-t^2}$. Now $\varphi(a)=\frac{1}{2}(\chi_+f_+(a)+\chi_-f_-(a))$  says that $C^*(a)$, which is of course both separable and commutative, is globally invariant under $\varphi$. Therefore, the endomorphism restricts to $C^*(a)$ trivially. This concludes the proof as $a\in\FA_{sa}$ was an arbitrary self-adjoint contraction.
\end{remark}

The following result is a generalization of  \cite{ACR}*{Theorem~6.16}.

\begin{thm}\label{thm:quasi-free-max-1}
Suppose that $(\gpt)$ is an algebraic dynamical system such that $G$ is abelian. If $C^*(G) \subset \CQ(\gxp)$ is maximal abelian and $\CQ(\gxp)$ is simple, then the group ${\rm Aut}_{C^*(G)}\CQ(\gxp)$ is maximal abelian in ${\rm Aut}\CQ(\gxp)$.
\end{thm}
\begin{proof}
Since $C^*(G) \subset \CQ(\gxp)$ is maximal abelian, we can apply Theorem~\ref{thm:aut of Q fixing C*(G)} to deduce that ${\rm Aut}_{C^*(G)}\CQ(\gxp) \cong {\rm Hom}_\theta(P,C(\widehat{G},\IT))$, which is abelian. Let $\alpha$ be such that $\alpha \circ\beta=\beta \circ\alpha$ for all $\beta\in {\rm Aut}_{C^*(G)}\CQ(\gxp)$. In particular, for $\beta = \Ad(x)$ with $x\in \CU(C^*(G))$ we get $\Ad(\alpha(x))=\alpha \circ\Ad(x)\circ \alpha^{-1}= \Ad(x)$. Thus we have $\Ad(x^{-1}\alpha(x)) = \id$, which implies that $x^{-1}\alpha(x)$ belongs to the center of $\CQ(\gxp)$. As $\CQ(\gxp)$ is simple, its center is trivial, so there is $\lambda_x \in \IT$ with $\alpha(x) = \lambda_x x$ for every $x \in \CU(C^*(G))$. Since $\alpha$ is a $*$-homomorphism, the values $\lambda_x$ arise from a character $\chi \in \widehat{G}$ via $\alpha(u_g) = \chi(g)u_g$ for all $g \in G$. Therefore, reinterpreting $\alpha(x) = \lambda_x x$ for $x \in \CU(C^*(G))$ as an equation for functions $f \in C(\widehat{G},\IT) \cong \CU(C^*(G))$, we arrive at
\[f(\chi \xi) = \lambda_f~f(\xi) \quad \text{for all } \xi \in \widehat{G}, f \in C(\widehat{G},\IT).\]
Since $G$ is countable, its compact abelian Pontryagin dual $\widehat{G}$ is a metrizable compact abelian group. Thus Lemma~\ref{lem:metgroup} implies $\chi=1_{\widehat{G}}$ so that $\alpha \in {\rm Aut}_{C^*(G)}\CQ(\gxp)$.
\end{proof}

We have chosen this abstract formulation of Theorem~\ref{thm:quasi-free-max-1} as we do not have a complete answer to the question when $C^*(G)$ is maximal abelian in $\CQ(\gxp)$ for arbitrary algebraic dynamical systems, and the precise condition for simplicity is somewhat technical in the greatest generality, see for instance \cites{Star,BS1}. Yet the proof we provide here does not need any of the relatively strong extra hypotheses that enter the proof of Theorem~\ref{thm:C*(G) masa}.

\begin{corollary}\label{cor:quasi-max-free}
If $(\gpt)$ belongs to the class described in Example~\ref{ex:C*(G) masa setup} and $\CQ(\gxp)$ is simple, then ${\rm Aut}_{C^*(G)}\CQ(\gxp)$ is maximal abelian in ${\rm Aut}\CQ(\gxp)$.
\end{corollary}
\begin{proof}
In the case of Example~\ref{ex:integral dynamics}, $\CQ(\gxp)$ is known to be simple and $(\gpt)$ satisfies the conditions of Example~\ref{ex:C*(G) masa setup}. Thus $C^*(G)$ is maximal abelian in $\CQ(\gxp)$ by Theorem~\ref{thm:C*(G) masa}, and Theorem~\ref{thm:quasi-free-max-1} yields the claim.
\end{proof}

As a special case of Corollary~\ref{cor:quasi-max-free}, we would like to highlight the case of integral dynamics:

\begin{corollary}\label{cor:quasi-max-free integral dynamics}
Suppose that $(\gpt)$ satisfies the assumptions of Example~\ref{ex:integral dynamics}. Then ${\rm Aut}_{C^*(\IZ)}\CQ(\IZ\rtimes P) \cong {\rm Hom}_\theta(P,C(\IT,\IT))$ is a maximal abelian subgroup of ${\rm Aut}\CQ(\IZ\rtimes P)$. 
\end{corollary}

\begin{remark}
We would like to point out that a general statement can be made here about arbitrary inclusions of $C^*$-algebras $\FB\subset\FA$ with $\FB$ is abelian and $\FA$ having trivial center. Indeed, one has the equality $C_{\Aut_{\FB}\FA}(\Aut\FA)=\CZ(\Aut_{\FB}\FA)$, where $C_{\Aut_{\FB}\FA}(\Aut\FA)$ is the centralizer of $\Aut_{\FB}\FA$ in $\Aut\FA$ and $\CZ(\Aut_{\FB}\FA)$ the center of $\Aut_{\FB}\FA$. As a result, the group $\Aut_{\FB}\FA$ is \emph{maximal} abelian whenever it is abelian.
\end{remark}

\subsection{A closer look at integral dynamics}\label{subsec:integral dynamics}
In this subsection, we assume that $(\gpt)$ is of the form specified in Example~\ref{ex:integral dynamics} so that we are dealing with subdynamics of $\Z\rtimes \N^\times$. In this situation, \eqref{eq:funct equation abstract} takes the more explicit form
\begin{equation}\label{eq:functional-equation} 
\psi_p(z)\psi_q(z^p)= \psi_q(z)\psi_p(z^q) \quad  \text{for all } z\in \IT \text{ and } p,q \in P.
\end{equation}
We observe that if one chooses $\psi_p(z)=f(z)$ for all generators $p \in S$ of $P$ for some $f \in C(\IT,\IT)$, then $f(z)=c$ for some constant $c\in\IT$ unless $\lvert S \rvert =1$.

It is not difficult to exhibit non-trivial solutions of this system of functional equations, but to the best of the authors' knowledge a complete description of all solutions of \eqref{eq:functional-equation} has yet to be found.

The largest specimen for Example~\ref{ex:integral dynamics} yields the $C^*$-algebra $\CQ_{\IN}$ introduced in \cite{CuntzQ}, that is, we consider the case $P=\N^\times$. For every other integral dynamics, the semidirect product $\Z\rtimes P$ embeds into $\Z\rtimes \N^\times$, along with an embedding of the corresponding $C^*$-algebra $\CQ(\Z \rtimes_\theta P)$ into $\CQ(\Z\rtimes \N^\times)\cong \CQ_\N$.
 
Cuntz remarked in \cite{CuntzQ}*{Section~4} that $\CQ_{\IN}$ is acted upon by a one-parameter group of automorphisms, given by $\gamma_t(u)\doteq u$ and $\gamma_t(s_p)\doteq p^{it}s_p$ for $p\in\IN^\times, t \in \R$, to which he refers as the \emph{canonical action} of $\IR$. Here we consider the slightly larger class of automorphisms ${\rm Aut}_{C^*(\Z)} \CQ(\Z\rtimes P)$. Due to Theorem~\ref{thm:C*(G) masa}, which applies to $\Z\rtimes P$, Theorem~\ref{thm:aut of Q fixing C*(G)} implies that we are dealing with automorphisms of the form $\beta_\psi$ for $\psi \in {\rm Hom}_\theta(P,C(\IT,\IT))$.

\begin{proposition}\label{prop:aut from unitary functions - monomial solutions}
Let $\psi \in {\rm Hom}_\theta(P,C(\IT,\IT))$. If $\psi_p$ is a monomial in the generator $z$ for all $p \in P$, then it is of the form $\psi_p(z) =z^{k(p-1)}$ for some $k \in \Z$.
\end{proposition}
\begin{proof}
Indeed, if we set $\psi_p(z)\doteq z^{k_p}$, then \eqref{eq:functional-equation} turns into $k_{pq}=k_p+pk_q$ for all $p,q \in \N^\times$. It is then straightforward to check that all solutions of this system of integral equations are of the form $k_p = k(p-1)$ for all $p$, where $k \in \Z$ is arbitrary. 
\end{proof}

The automorphisms $\beta_{\psi_k}$ arising from the monomial solutions obtained in Proposition~\ref{prop:aut from unitary functions - monomial solutions} are inner as $\beta_{\psi_k}=\Ad(u^{-k})$.

\begin{proposition}\label{prop:aut from unitary functions - constant solidarity}
Let $\psi \in {\rm Hom}_\theta(P,C(\IT,\IT))$. If there is $n\geq 2$ such that $\psi_n$ is constant, then $\psi_p$ is constant for all $p \in P$. 
\end{proposition}
\begin{proof}
By \eqref{eq:functional-equation}, we get $\psi_p(z)\psi_n(z^p)=\psi_n(z)\psi_p(z^n)$, which reduces to $\psi_p(z)=\psi_p(z^n)$ for all $p\in P$ as $\psi_n$ is constant. But the only solutions to this type of functional equation are constant since $n \geq 2$.
\end{proof}

In the second part of this final section, we show that every automorphisms $\alpha$ of $\CQ(\Z\rtimes P)$ that fixes a natural subalgebra $\CO_n = C^*(\{u^ks_n \mid 0 \leq k \leq n-1\})$ pointwise for some $n \in P, n \geq 2$ is necessarily of the form $\beta_\psi$ appearing in Theorem~\ref{thm:aut of Q fixing C*(G)} for some $\psi \in {\rm Hom}(P,\IT)$, so that, in particular, $\alpha$ fixes $C^*(\Z)$ pointwise.

\begin{thm}\label{thm:aut of Q_N fixing O_n}
For every $n \in P$, $2\leq n < \infty$, the group ${\rm Aut}_{\CO_n}\CQ(\Z\rtimes P)$ is a subgroup of ${\rm Aut}_{C^*(\Z)}\CQ(\Z\rtimes P) \cong {\rm Hom}_\theta(P,C(\IT,\IT))$. For $\psi \in {\rm Hom}_\theta(P,C(\IT,\IT))$, the automorphism $\beta_\psi$ belongs to ${\rm Aut}_{\CO_n}\CQ(\Z\rtimes P)$ if and only if $\psi \in {\rm Hom}(P,\IT)$ with $\psi_n=1$, that is, $\beta_\psi = \gamma_\chi$ for some $\chi \in T$ with $\chi(n) = 1$.
\end{thm}
\begin{proof}
Let $n \in P$, $2\leq n < \infty$ and $\alpha \in {\rm Aut}_{\CO_n}\CQ(\Z\rtimes P)$. We work in the canonical representation $\pi\colon \CQ(\Z\rtimes P) \to \CL(\ell^2(\Z))$ to show that $\alpha(u)=u$. Noting that $e_{m,n^k} = u^ms_{n^k}^{\phantom{*}}s_{n^k}^*u^{-m} \in \CO_n$ for $k \geq 1$ if 
 $0\leq m \leq n^k-1$, we get 
\[\alpha(u) e_{m,n^k} = \alpha(u e_{m,n^k}) = u^{m+1}s_{n^k}^{\phantom{*}}s_{n^k}^*u^{-m} = u e_{m,n^k}\]
for all $0 \leq m \leq n^k-2$. Therefore, we obtain $\pi(\alpha(u))$ and $\pi(u)$ coincide outside $\ell^2(\bigcap_{k \geq 1} n^k-1+n^k\Z) = \C\xi_{-1} \subset \ell^2(\Z)$. Now $\CQ(\Z\rtimes P)$ is purely infinite and simple, so its image under $\pi$ does not contain any compact operator other than $0$. Thus faithfulness of $\pi$ allows us to conclude that $\alpha(u)=u$, so that ${\rm Aut}_{\CO_n}\CQ(\Z\rtimes P)$ is a subgroup of ${\rm Aut}_{C^*(\Z)}\CQ(\Z\rtimes P) \cong {\rm Hom}(P,C(\IT,\IT))$, where we refer to Theorem~\ref{thm:aut of Q fixing C*(G)} for the latter identification.

Now suppose $\psi \in {\rm Hom}_\theta(P,C(\IT,\IT))$ satisfies $\beta_\psi \in {\rm Aut}_{\CO_n}\CQ(\Z\rtimes P)$. Then $s_n = \beta_\psi(s_n) = \psi_n s_n$ forces $\psi_n=1$ as in the proof for injectivity in Theorem~\ref{thm:aut of Q fixing C*(G)}. But then Proposition~\ref{prop:aut from unitary functions - constant solidarity} implies $\psi_p \in \IT$ for all $p \in P$, so that $\psi \in {\rm Hom}(P,\IT)$ and hence $\beta_\psi$ is a gauge automorphism $\gamma_\chi$. 
\end{proof}

This result has an interesting immediate consequence for which we recall from Example~\ref{ex:integral dynamics} that $S\subset \N^\times$ is a family of mutually coprime numbers and $P$ is the submonoid of $\N^\times$ generated by $S$. 

\begin{corollary}\label{cor:aut fixing torsion subalgebra}
For $\emptyset \neq S' \subset S$, let $Q\doteq \langle S\setminus S'\rangle \subset \N^\times$, $R \doteq \widehat{Q^{-1}Q}$, and $\CA_{S'} \doteq C^*(\{u^ks_p \mid p \in S', 0\leq k\leq p-1\}) \subset \CQ(\Z\rtimes P)$. Then ${\rm Aut}_{\CA_{S'}}\CQ(\Z\rtimes P)$ is isomorphic to the group of gauge automorphisms of $R$ on $\CQ(\Z \rtimes Q)$. In particular, the only automorphism of $\CQ(\Z\rtimes P)$ that fixes the torsion subalgebra $\CA_S$ is the identity.
\end{corollary}

Corollary~\ref{cor:aut fixing torsion subalgebra} is remarkable as the torsion subalgebra plays a central role for the K-theory of $\CQ(\Z\rtimes P)$. We remark that in view of 
\[\CQ(\Z\rtimes P) \cong \varinjlim M_p(C^*(\Z)) \rtimes P \quad \text{whereas} \quad \CA_S \cong \varinjlim M_p(\C) \rtimes P,\] 
this rigidity is already apparent from Theorem~\ref{thm:aut of Q_N fixing O_n} because it tells us that the group $C^*$-algebra has to be fixed pointwise as well.

\begin{remark}
Thanks to Theorem~\ref{thm:aut of Q_N fixing O_n}, it is now easy to see that the group  $\Aut_{\CO_n}\CQ(\Z\rtimes P)$ is topologically isomorphic with $\IT^{\lvert S \rvert-1}$, where $\lvert S \rvert$ is the rank of $P$, see Example~\ref{ex:integral dynamics}, irrespective of what $n$ is. Here, $\IT^{\lvert S \rvert-1}$ is equipped with the product topology and $\Aut_{\CO_n}\CQ(\Z\rtimes P)$ with the topology of pointwise convergence in norm.   
\end{remark}

\subsection{Outerness}\label{subsec:outerness}
We start with a simple observation for ${\rm Aut}_{C^*(G)}\CQ(\gxp)$ as discussed in Subsection~\ref{subsec:aut fix C*(G)} in the case of $P = \N$, where every $\psi \in {\rm Hom}_\theta(P,C(\widehat{G},\IT))$ is determined by a single unitary $f \in C(\widehat{G},\IT)$. By writing $\beta_f$ for $\beta_\psi$, we obtain the following generalization of \cite{ACR}*{Proposition 6.5}:

\begin{proposition} \label{prop:outer condition for beta_F}
Let $(\gpt)$ be an algebraic dynamical system as in Example~\ref{ex:abelian group - single finite endo}. If $f \in C(\widehat{G},\IT)$ satisfies $f(1_{\widehat{G}})\neq 1$, then $\beta_f$ is an outer automorphism.
\end{proposition}
\begin{proof}
In the following computations, we work in the representation obtained out of the canonical representation via Fourier transform. This is the representation on $L^2(\widehat{G},\mu)$ given by $\pi(s)\xi(\chi)\doteq \xi(\hat{\theta}(\chi))$ for $\chi\in\widehat{G}$, where $\hat{\theta}$ is the surjective group endomorphism of $\widehat{G}$ corresponding $\theta$ via Pontryagin duality, $\pi(u_g)\xi(\chi)\doteq \chi(u_g) \xi(\chi)$, $\xi\in L^2(\widehat{G},\mu)$.

Suppose that $\beta_F = {\rm Ad}(V)$ for some $V\in \CU(\CQ(G\rtimes \IN))$. Since $C^*(G)$ is a maximal abelian subalgebra of $\CQ(G\rtimes \IN)$, the unitary $V$ is in $C^*(G)\cong C(\widehat{G},\IT)$, which means that $V=M_g$ with $g\in C(\widehat{G},\IT)$. In our representation the equality $f(u)s=VsV^*$ reads as
\begin{align*}
f(\chi) \xi(\hat{\theta} (\chi)) & =f(\chi)(s\xi)(\chi)=(M_f s)(\xi)(\chi)=(M_gsM_{\overline{g}})(\xi)(\chi)=g(\chi)(sM_{\overline{g}})(\xi)(\chi)\\
& =g(\chi)(M_{\overline{g}})(\xi)(\hat{\theta} (\chi))=g(\chi)\overline{g(\hat{\theta} (\chi))}\xi(\hat{\theta} (\chi))
\end{align*}
In particular, if we choose $\chi = \xi = 1_{\widehat{G}}$ we get
\[
f(1_{\widehat{G}})=g(1_{\widehat{G}})\overline{g(\hat{\theta} (1_{\widehat{G}}))}=g(1_{\widehat{G}})\overline{g}(1_{\widehat{G}})=1.
\]
\end{proof}

The next result is a generalization of \cite{ACR}*{Theorem~5.1} for $\CQ_2$, but the proof we give here is quite different. Recall that $N_p = [G:\theta_p(G)]$, and if $P$ is an abelian right LCM semigroup, we denote by $T$ the Pontryagin dual of its Grothendieck group $P^{-1}P$ and by $\gamma\colon T \curvearrowright \CQ(\gxp)$ the gauge action given by $\gamma_\chi(u_gs_p) = \chi(p)u_gs_p$ for $\chi \in T, g \in G, p\in P$. For $(\gpt)$ with abelian $G$, we denote by $\widehat{\theta_p}$ the surjective group endomorphism of $\widehat{G}$ corresponding to $\theta_p$ for $p \in P$.

\begin{thm}\label{thm:gaugeQpouter}
Suppose that $(\gpt)$ is an algebraic dynamical system with $G$ and $P$ abelian, and $2 \leq [G:\theta_p(G)] < \infty$ for some $p \in P$. If $C^*(G) \subset \CQ(\gxp)$ is maximal abelian and the only $\chi \in \widehat{G}$ for which the system of functional equations 
\begin{equation}\label{eq:funct eq - outer gauge action}
\chi(p)f = f \circ \widehat{\theta_p} \quad \text{for all } p \in P \text{ with } N_p<\infty
\end{equation}
admits a solution $f$ in $C(\widehat{G},\IT)$ is $\chi=1$, then $\gamma\colon T \curvearrowright \CQ(\gxp)$ is an outer action.
\end{thm}  
\begin{proof}
Let $\chi \in \widehat{G}$ such that $\gamma_\chi = {\rm Ad}(w^*)$ for some $w \in \CU(\CQ(\gxp))$. As $u_g = \gamma_\chi(u_g) = w^*u_gw$ for all $g \in G$, we get $w \in \CU(C^*(G)) \cong C(\widehat{G},\IT)$ because $C^*(G)$ is maximal abelian in $\CQ(\gxp)$. If $f \in C(\widehat{G},\IT)$ corresponds to $w$, then relation (I) yields 
\[\chi(p)s_p = \gamma_\chi(s_p) = \overline{f} s_p f = \overline{f}(f \circ \widehat{\theta_p})s_p \quad \text{for all } p \in P.\]
Multiplying the above equations by $f$ on the left and using (III) for $p \in P$ with $N_p < \infty$, we arrive at \eqref{eq:funct eq - outer gauge action}. By assumption, this forces $\chi=1$.
\end{proof}

\begin{corollary}\label{cor:gaugeQouter integral dynamics}
If $(\Z,P,\theta)$ is an integral dynamics as described in Example~\ref{ex:integral dynamics}, then $\gamma\colon \IT^{|S|} \curvearrowright \CQ(\Z \rtimes P)$ is outer.
\end{corollary}
\begin{proof}
Note that $G$ and $P$ are abelian, the system is of finite type, and $C^*(\Z)\subset \CQ(\Z\rtimes P)$ is maximal abelian due to Theorem~\ref{thm:C*(G) masa}, see the discussion following Example~\ref{ex:C*(G) masa setup}. For $w \in T=\IT^{|S|}$, \eqref{eq:funct eq - outer gauge action} gives $w_p f(z) = f(z^p)$ for all $z \in \IT$ and all $p \in S$. But it is well known that each of these individual equations only has a solution for $w_p=1$, so that $w=1$. Now the result follows from Theorem~\ref{thm:gaugeQpouter}.
\end{proof}

In fact, Corollary~\ref{cor:gaugeQouter integral dynamics} can be easily deduced from the following result that competes with Theorem~\ref{thm:gaugeQpouter}:

\begin{corollary}\label{cor:gaugeouter}
If $(\gpt)$ is an algebraic dynamical system for which $P$ is abelian, $P^*$ is finite, and $\bigcap_{p \in P}\theta_p(G) = \{1_G\}$, then $\gamma\colon T \curvearrowright \CQ(\gxp)$ is an outer action.
\end{corollary}
\begin{proof}
Suppose $\gamma_\chi$ is inner, that is, $\gamma_\chi=\Ad{w}$, for some unitary $w\in\CQ(\gxp)$. Then Lemma \ref{lem:unitaries in rel commutant} applies as $ws_pw^* = \gamma_\chi(s_p) = \chi(p)s_p$ with $\chi(p) \in \IT$, and implies $\chi=1$.
\end{proof}

Our choice to include both Theorem~\ref{thm:gaugeQpouter} and Corollary~\ref{cor:gaugeouter} is based on the difference in their approaches: While Theorem~\ref{thm:gaugeQpouter} makes seemingly restrictive assumptions like $C^*(G) \subset \CQ(\gxp)$ being maximal abelian and the non-existence of non-trivial solutions to \eqref{eq:funct eq - outer gauge action}, it does not require $\bigcap_{p \in P}\theta_p(G) = \{1_G\}$, even though our Theorem~\ref{thm:C*(G) masa} needs this assumption. On the other hand, Corollary~\ref{cor:gaugeouter} only asks for $P$ to be abelian, finiteness of $P^*$, and $\bigcap_{p \in P}\theta_p(G) = \{1_G\}$. So unless the assumptions in Theorem~\ref{thm:gaugeQpouter} force $\bigcap_{p \in P}\theta_p(G) = \{1_G\}$ and $P^*$ to be finite, the two results remain independent.\vspace*{3mm}

We stay with integral dynamics as described in Example~\ref{ex:integral dynamics} and prove outerness of all automorphisms of $\alpha \in {\rm Aut}\CQ(\Z\rtimes P)$ that invert the unitary $u$, that is, $\alpha(u) = u^*$, see Theorem~\ref{thm:aut of Q_N inverting u are outer}. The two particular examples we have in mind as motivation are
\begin{enumerate}[(a)]
\item $\lambda_{-1}$ given by $u \mapsto u^*, s_p \mapsto s_p$; and
\item $\phi$ given by $u \mapsto u^*, s_p \mapsto u^{p-1}s_p$ for $p \in P$.
\end{enumerate}
In order to check that $\phi$ is a morphism of $\CQ_\IN$ we only need to check that $\phi(s_p)\phi(s_q)=\phi(s_{pq})$ (II) holds, which follows from 
\[\phi(s_p)\phi(s_q) = u^{p-1}s_pu^{q-1}s_q = u^{p-1+p(q-1)}s_{pq}= \phi(s_{pq}).\] 
The reason is that the relations (I) and (II) for $\CQ(\Z\rtimes P)$ are satisfied by $\phi(u)$ and $\phi(s_p)$ for all $p \in P$ because the restriction $\phi |_{\CQ_p}$ coincides with the unique extension of the exchange automorphism on $\CQ_p \subset \CQ(\Z\rtimes P)$, see Definition~\ref{def:exchange automorphism} and Theorem~\ref{thm:extensible Bogolubov autos}. In fact, $\phi$ is determined completely by the collection of all these exchange automorphisms. We note that $\phi$ and $\lambda_{-1}$ are unitarily equivalent as $\phi =\Ad(u^*)\circ\lambda_{-1}$.

\begin{thm}\label{thm:aut of Q_N inverting u are outer}
Every automorphism $\alpha$ of $\CQ(\Z\rtimes P)$ with the property $\alpha(u)=u^*$ is outer. In particular, $\lambda_{-1}$ and $\phi$ are outer.
\end{thm}
\begin{proof}
The proof is an adaptation of the proof of \cite{ACR}*{Theorem~5.9}, so we rather explain the strategy and point out modifications than go through the proof in full detail: The symmetry $\IP \in \CL(\ell^2(\Z))$ given by $\IP\xi_n \doteq \xi_{-n}$ for $n \in \Z$ satisfies $\IP\pi(u)\IP = \pi(u^*)$, see \cite{ACR}*{Remark~5.5}. Then \cite{ACR}*{Theorem~5.8} shows that $\IP \notin \CQ_2$, which we shall explain here for $\CQ(\Z\rtimes P)$: The $C^*$-algebra $\CQ(\Z\rtimes P)$ is the closure of the linear span of operators of the form $u^m s_p^{\phantom{*}}s_q^*u^{*n}$ with $p,q \in P, m,n \in \Z$. So let $\sum_{1 \leq i \leq j} c_i u^{m_i}s_{p_i}^{\phantom{*}}s_{q_i}^*u^{n_i}$ with $c_i \in \C, m_i,n_i \in \Z, p_i,q_i \in P$ for $i=1,\ldots,j$ for some $j \in \N$. For every $n \geq \max_{1 \leq i\leq j} \lvert m_i\rvert \vee \lvert n_i\rvert$ we then get
\[\begin{array}{c}
\bigl\| \bigl(\IP-\sum\limits_{1 \leq i \leq j} c_i u^{m_i}s_{p_i}^{\phantom{*}}s_{q_i}^*u^{n_i}\bigr) \xi_n\bigr\| = \bigl\|\xi_{-n} - \sum\limits_{1 \leq i \leq j} c_i u^{m_i}s_{p_i}^{\phantom{*}}s_{q_i}^*u^{n_i}(\xi_n)\bigr\| \geq 1
\end{array}
\]
as 
\[ u^{m_i}s_{p_i}^{\phantom{*}}s_{q_i}^*u^{n_i}(\xi_n) =
\begin{cases}
\xi_{m_i+\frac{p_i}{q_i}(n_i+n)} &, \text{if } n_i+n \in q_i\Z,\\
0 &, \text{otherwise.}
\end{cases}\]
The reason is that the choice of $n$ and $P \subset \N^\times$ force $\frac{p_i}{q_i}(n_i+n)$ to have the same sign as $n$. Thus, again by the choice of $n$, we have $m_i+\frac{p_i}{q_i}(n_i+n) \neq -n$ for all $i$, and therefore $\IP \notin \CQ(\Z\rtimes P)$. 

Now suppose $\alpha \in \Aut \CQ(\Z\rtimes P)$ with $\alpha(u)=u^*$ is inner, that is, there is $w \in \CU(\CQ(\Z\rtimes P))$ such that $\alpha = {\rm Ad}(w)$. Then we get $\IP\pi(wuw^*)\IP = \pi(u)$, so that the unitary $\IP\pi(w)$ commutes with $\pi(u)$. Thus we deduce $\IP\pi(w) \in \CU(\pi(C^*(\Z))') \cong \CU(L^\infty(\IT,\mu)) = L^\infty((\IT,\mu),\IT)$, where $\mu$ denotes the Haar measure (or Lebesque measure) on $\IT$. Since $\IP$ is a symmetry, we arrive at $\pi(w) \in \IP \cdot L^\infty((\IT,\mu),\IT)$. If we could take $f \in C(\IT,\IT) \subset L^\infty((\IT,\mu),\IT)$ with $\pi(w) = \IP \pi(f(u))$ instead of an essentially bounded function, the contradiction would follow readily as $\IP = \IP \pi(f(u)\overline{f}(u)) = \pi(w\overline{f}(u)) \in \CQ(\Z\rtimes P)$ contradicts $\IP \notin \CQ(\Z\rtimes P)$. However, the general case is more technical, and we refer to the proof of \cite{ACR}*{Theorem~5.9} for details, remarking only some constants within the estimations will change.
\end{proof}

\begin{remark}\label{rem:outerness for inverting u beyond}
It is crucial to Theorem~\ref{thm:aut of Q_N inverting u are outer} that there is no $p \in \N^\times$ such that the monoid $P\subset \Z^\times$ contains both $p$ and $-p$ because then $\IP = \sum_{0 \leq m \leq p-1}\pi(u^ms_{-p}^{\phantom{*}}s_p^*u^{-m})$ belongs to $\pi(\CQ(\Z\rtimes P)) \cong \CQ(\Z\rtimes P)$, and the proof of Theorem~\ref{thm:aut of Q_N inverting u are outer} breaks down. In fact, we have $\pi \circ \lambda_{-1} = {\rm Ad}(\IP) \circ \pi$, and thus both $\lambda_{-1}$ and also $\phi = {\rm Ad}(u^*)\circ \lambda_{-1}$ are inner. One may therefore ask for new examples of outer automorphisms of $\CQ(\Z \rtimes P)$, especially with an eye on the case where $P \subset \Z^\times$ contains $-1$. 
\end{remark}

\begin{corollary}\label{cor:extensible Bogolubov autom of O_p give outer action}
For every $2\leq p < \infty$, the group of extendible Bogolubov automorphisms of $\CO_p$ defines an outer action of $\IT \times \Z/2\Z$ on $\CQ_p$.
\end{corollary}
\begin{proof}
This follows immediately from combining Theorem~\ref{thm:extensible Bogolubov autos} with Corollary~\ref{cor:gaugeQouter integral dynamics} and Theorem~\ref{thm:aut of Q_N inverting u are outer}.
\end{proof}

\end{document}